\documentclass{ajour}
\usepackage{latexsym,amsfonts,verbatim}
\usepackage{amssymb}
\usepackage{amsmath}

\def\Mat{\mathop{\rm Mat}\nolimits}          
\def\GL{\mathop{\rm GL}\nolimits}

\def\SO{\mathop{\rm SO}\nolimits}
\def\SL{\mathop{\rm SL}\nolimits}
\def\GL{\mathop{\rm GL}\nolimits}            
\def\PSL{\mathop{\rm PSL}\nolimits}          
\def\PSp{\mathop{\rm PSp}\nolimits}           
\def\PGL{\mathop{\rm PGL}\nolimits}           
\def\SU{\mathop{\rm SU}\nolimits}          
\def\PSU{\mathop{\rm PSU}\nolimits}          
\def\Sp{\mathop{\rm Sp}\nolimits}
\def\CSp{\mathop{\rm CSp}\nolimits}
\def\CO{\mathop{\rm CO}\nolimits}
\def\CU{\mathop{\rm CU}\nolimits}

\def\Sym{\mathop{\rm Sym}\nolimits}

\newcommand{\tr}[1]{{\rm Tr}(#1)}
\def\char{\rm char\,}

\def\dim{\ \mathop{\rm dim}\nolimits\ }

\def\Alt{\mathop{\rm Alt}\nolimits}

\def\id{\mathrm{id}}
\def\blockdiag{\mathop{\rm blockdiag}\nolimits}
\def\diag{\mathop{\rm diag}\nolimits}

\newcommand{\N}{\mathbb{N}}    
\newcommand{\F}{\mathbb{F}}    

\begin{document}
\title{Uniform $(2,k)$-generation of the 4-dimensional classical groups} 

\author{M. A. Pellegrini}
\affil{ Dipartimento di Matematica e Applicazioni, Universit\`a degli Studi di Milano-Bicocca\\ 
  Via R. Cozzi, 53, 20125 Milano Italy}
\email{marco.pellegrini@unimib.it}

\author{M. C. Tamburini Bellani}
\affil{Dipartimento di Matematica e Fisica, Universit\`a Cattolica del Sacro Cuore\\
 Via Musei, 41, 25121 Brescia Italy}
 \email{c.tamburini@dmf.unicatt.it}

\author{M. A. Vsemirnov\footnote{The third author was supported by the Cariplo Foundation, the Russian
Foundation for Basic Research (grant no.~09-01-00784-a) and the
Federal Target Programme ``Scientific and scientific-pedagogical
personnel of the innovative Russia 2009-2013'' (contract no.~P265).}
}
\affil{St.Petersburg Division of Steklov Institute of Mathematics\\
  27, Fontanka, 191023, St.Petersburg, Russia}
\email{vsemir@pdmi.ras.ru}

\abstract{In this paper we study the $(2,k)$-generation of the finite classical groups
$\SL_4(q)$, $\Sp_4(q)$, $\SU_4(q^2)$ and their projective images. Here $k$ is the order of 
an arbitrary element of $\SL_2(q)$, subject to the necessary condition $k\ge 3$. 
When $q$ is even we allow also $k=4$. }

\section{Introduction}

We recall that a group is $(2,k)$-generated if it can be generated by two
elements of respective orders $2$ and $k$. Our aim is to find uniform
$(2,k)$-generators of the 4-dimensional classical groups.
Since two involutions generate a dihedral group, we assume $k\geq 3$.

In this problem a special case of a formula of L. Scott \cite[Theorem 1]{Sc} plays a crucial role. 
E.g. it gives constraints for the similarity invariants  of a $(2,k)$-generating pair
of the groups under consideration (see \eqref{one}). Moreover it gives a rigidity criterion,
proved by Strambach and V\"olklein \cite[Theorem 2.3]{SV}, which is very useful (see \eqref{rigidity}).
We acknowledge A. Zalesskii for having brought our attention to this subject
and to a systematic study of Scott's result.
Sections \ref{Scott's formula} and \ref{preserving}
are dedicated to this study, aiming to develop more general techniques for 
linear groups. 

In Section \ref{shape} we fix the canonical forms of 
our uniform generators $x,y$. Having in mind the $(2,k)$-generation of the 
projective images of the classical groups, we allow 
$x^2=\pm I$. On the other hand both $y$ and its projective image have order $k$. 
The choice of the canonical form of $y$ is determined, for uniformity reasons,
by the case $k=3$ (see the beginning of Section \ref{negres}). 
We characterize the shapes of $x$ and $y$, up to conjugation, subject to the condition
that the group $\left\langle x,y\right\rangle$ is absolutely irreducible (see \eqref{generators}).
The matrix $x$ is uniquely determined by its order,
$y$ has an entry $s$ which determines its order $k$ and 
four indeterminate entries $r_1,\dots , r_4$: their values 
which still produce a reducible group
are described by Lemma \ref{reducible}.

In Section \ref{negres}, we are ready to prove a list 
of negative results. They show that,  
apart from the groups $\Sp_4(q)$, which probably require
a generator of order $k\geq 4$ with similarity invariants other 
than  $y$, our positive results are the best possible. 
In particular the following groups are not $(2,3)$-generated: 
$\SL_4(2)$, $\Sp_4(q)$ for all $q$, $\PSp_4(2^a)$, $\PSp_4(3^a)$, $\SU_4(9)$ and 
$\PSU_4(9)$.
Moreover $\SL_4(3)$, $\SU_4(9)$ and $\PSU_4(9)$ are not $(2,4)$-generated:
this last fact has required an unexpected amount of details.
The exception of symplectic groups 
were detected by Liebeck and Shalev in \cite{LS}.
Here we give an alternative proof.

In Section \ref{sec:further} we specialize the values of the parameters
$r_1,\dots , r_4$, aiming to our positive results.  Let
$\F$ be an algebraically closed field of characteristic $p>0$.
If $(k,p)=1$, let us denote by $\epsilon$ a primitive $k$-th root of unity in $\F$.
If $k=p$ or $k=2p$, we set respectively  $\epsilon=1$ or $\epsilon =-1$. 
Writing $s=\epsilon+\epsilon^{-1}$, our uniform generators have shapes:
$$
x=\left(
\begin{array}{cccc}
0&0&1&0\\
0&0&0&1\\
d&0&0&0\\
0&d&0&0
\end{array}\right),\ d=\pm 1, \qquad
y=\left(
\begin{array}{cccc}
1&0&0&r_2\\
0&1&0&r_4\\
0&0&0&-1\\
0&0&1&s
\end{array}\right),\ r_4\ne 0.
$$
By Corollary \ref{newreducible} the group  $H=\left\langle x,y\right\rangle$ is absolutely irreducible 
provided $r_2\ne -\epsilon^{\pm 1} r_4$
and $r_2 + r_4\ne\pm (2- s)\sqrt d$. 
We make the assumption that $H$ is absolutely irreducible, which 
has several consequences.

First of all, by  Remark \ref{rigid triple}, the triple $(x,y,xy)$ is rigid and
$xy$ must have a unique similarity invariant. 
It follows that, for any field automorphism $\sigma$, 
the matrices $(xy)^\sigma$ and $(xy)^{-1}$ are conjugate
if and only if their  characteristic polynomials are the same.
They are respectively:
$$\begin{array}{cccccc}
\chi_{(xy)^\sigma}(t)\hfill &=& t^4-dr_4^\sigma t^3-ds^\sigma t^2-r_2^\sigma t+1,\hfill  \\
\noalign{\smallskip}
\chi_{(xy)^{-1}}(t)\hfill &=& t^4-r_2t^3-dst^2-dr_4t+1.\hfill
\end{array}$$
Suppose that $r_2$, $r_4$ and $s$ belong to $\F_q$. If $H\leq \Sp_4(q)$,  
then $xy$ is conjugate to its inverse, and we obtain the necessary condition
$r_2=dr_4$. If $d=1$ and  $r_2=r_4$,  then
$H\leq \SO^{\pm}(q)$ by Theorem \ref{classic_positive}(ii).
This fact and Theorem \ref{nr}(iv) explain why, in the symplectic case, we 
must assume $d=-1$ and $p$ odd.
On the other hand, the conditions $d=-1$, $p$ odd and $r_2=-r_4$ 
are sufficient to guarantee that $H\leq \Sp_4(q)$
by Theorem \ref{classic_positive}(v).
Next suppose that $r_2, r_4\in \F_{q^2}$, $s\in \F_q$. If $H\leq \SU_4(q^2)$,  
taking the Frobenius map
$\alpha\mapsto \alpha^q$ as the field automorphism $\sigma$, we obtain the necessary condition $r_2=dr_4^q$.
On the other hand, this condition  is sufficient to guarantee 
that $H$ is contained in $\SU_4(q^2)$ by Theorem \ref{forms}(ii).

Now suppose that the group $H$ is contained in one 
of the classical groups under consideration.
If $H$ is not the whole group, then it is contained in some maximal subgroup $M$.
Following Aschbacher's structure Theorem \cite{A84},
the group $M$ belongs to one of nine classes. Eight of them,  denoted from 
${\mathcal C_1}$ to ${\mathcal C_8}$, correspond to natural subgroups.
The remaining class $\mathcal{S}$ results from absolutely irreducible representations of finite simple groups.
So, for fixed $q$ and $s$,we have to exclude all values of $r_2$ and $r_4$ for which $H$ is 
contained in some $M$ in the above classes.
If $M\in {\mathcal C_1}\cup{\mathcal C_3}$,
then it is reducible over $\F$: thus Corollary \ref{newreducible} takes care of
this case. If  $M\in {\mathcal C_2}$, it stabilizes a direct sum decomposition: this
possibility is considered in Section \ref{C2}. If $M\in {\mathcal C_5}$, 
it stabilizes a subfield: see Lemma \ref{minimal field}. 
The case $M\in {\mathcal C_8}$, the class   of classical subgroups, is studied
in Theorem \ref{classic}. In dimension 4, this analysis 
includes also the  case $M\in {\mathcal C_4}\cup {\mathcal C_7}$.
Indeed, up to conjugation, such an $M$ is contained in ${\rm GL}_2(\mathbb{F})\otimes {\rm GL}_2(\mathbb{F})$: 
hence it fixes, up to a scalar,
the matrix $J\otimes J$, where $J=$ antidiag$(1,-1)$.
Finally, the cases $M\in {\mathcal C_6}$ and $M\in {\mathcal S}$ are considered
in Sections  \ref{C6} and  \ref{S} respectively.  The conditions obtained, except those given
by Lemma \ref{minimal field}, are summarized in
Table  \ref{tab:4}.

We are not aware of any published list of the maximal subgroups of the finite 4-dimensional
classical groups. So our reference was the Ph.D. thesis of Kleidman \cite{K}, whose list
is based on the work of several authors, namely \cite{F}, \cite{KL82}, \cite{M}, \cite{M76},  \cite{M82}, \cite{SZ}, \cite{Z}.

The last Section contains our positive results. 
Their precise statements, formulated in Theorems \ref{sl4}, \ref{sp4}, \ref{su4} and
\ref{thm:su4-2}, can be roughly summarized as follows. 
Up to a finite number of exceptions, completely determined, we have that:

\smallskip
$\bullet$ if $d=\pm 1$, $0=r_2$, $0\ne r_4\in \F_q$ and $\F_q=\F_p\!\left[s, r_4^2\right]$, then
$\left\langle x,y\right\rangle =\SL_4(q)$;

\smallskip
$\bullet$ if $d=-1$, $r_2=-r_4$, $p\ne 2$, $s\ne 2$, $0\ne r_4\in \F_q$ and $\F_q=\F_p\!\left[s, r_4^2\right]$, 
then $\left\langle x,y\right\rangle =\Sp_4(q)$;

\smallskip
$\bullet$ if $d=\pm 1$, $s\in \F_q$, $r_2=d r_4^q$,  
$0\ne r_4\in \F_{q^2}$ and $\F_{q^2}=\F_p\!\left[r_4^2\right]$, then
$\left\langle x,y\right\rangle =\SU_4(q^2)$.

\smallskip
As a consequence we obtain that, for all $k\ge 3$ such that
$k\mid (q-1)$ or $k\mid (q+1)$ or $k=p$ or $k=2p$, the following simple groups are $(2,k)$-generated:
$$\PSL_4(q),\ q>2,\quad \PSp_4(q),\ k\ne p,\ q\  {\rm odd},\quad \PSU_4(q^2),\ q>3.$$
In particular they are $(2,3)$-generated, 
except $\PSp_4(2^a)$ and $\PSp_4(3^a)$. 

Actually other papers establish explicitly the $(2,3)$-generation
of many of them. To our knowledge the groups $\PSL_4(q)$ were considered in
\cite{MT},  \cite{TV1}, \cite{TV2} 
and the groups $\PSp_4(q)$ were considered in \cite{CD}.

More recent work in a related area is due to Marion \cite{Ma}, who studies the groups  
$\PSL_n(q)$, $n\le 3$, which are epimorphic images of a given hyperbolic triangle group.

Finally we note that most of our results are computer independent. Nevertheless MAGMA and GAP have
been of great help in the computational aspects of this paper.

\section{Scott's formula and rigidity}\label{Scott's formula}

We recall some basic consequences of Scott's formula  \cite[Theorem 1]{Sc}: 
for their background we refer to \cite{TV}.  Here $\left\langle X,Y\right\rangle$ 
denotes  an absolutely irreducible sugbgroup  of $\GL_n(L)$, where $L$ is a field.
For a subset $K$ of $M=\Mat_n(L)$, let $d^K_M$ be the dimension of $C_M(K)$. 
Then, with respect to the conjugation action of $\left\langle X,Y\right\rangle$ on $M$,
Scott's formula gives the condition:
\begin{equation}\label{one}
d^X_M+d^Y_M+d^{XY}_M\ \le\ n^2+2.
\end{equation}
Moreover, if equality holds, namely if
\begin{equation}\label{rigidity}
d^X_M+d^Y_M+d^{XY}_M\ =\ n^2+2
\end{equation}
the triple $(X,Y,XY)$ is rigid \cite[Theorem 2.3]{SV}. This means that,
for any other triple $(X',Y',X'Y')$ with the same similarity invariants
as $(X,Y,XY)$, there exists $g\in \GL_n(L)$ such that $X'=X^g$ and
$Y'=Y^g$. In particular the groups $\left\langle X,Y\right\rangle$  and $\left\langle X',Y'\right\rangle$ 
are conjugate.

We may identify the space $L^n\otimes L^n$
with $M$, via the linear extension of the map
$e_i\otimes e_j\mapsto e_ie_j^T$. In particular 
the symmetric square $S$ of $L^n\otimes L^n$ is identified 
with the space of symmetric matrices. 
Clearly, for any $g\in K$ as above, 
the diagonal element $g\otimes g$ acts as $m\mapsto gmg^T$,  for all $m\in M$.
Now let  us denote respectively by $d_S^K$ and $\hat d_S^K$ 
the dimension of the space of $K$-fixed points on $S$ and on its dual.
In this case, Scott's formula gives the condition:
\begin{equation}\label{two}
d_S^X+d_S^Y+d_S^{XY}\leq \frac{n(n+1)}{2} +d_S^{\left\langle X,Y\right\rangle}+
 \hat d_S^{\left\langle X,Y\right\rangle}.
\end{equation}
By \cite[Lemma 1]{TV}, setting $K=\left\langle X,Y\right\rangle$, we have $d_S^{K}\le 1 $,
$\hat d_S^{K} \le 1$. If $\char L \ne 2$, then
$d_S^{K}=\hat d_S^{K}$.  
Moreover if $\hat d_S^{K} =1$,
then $d_S^{K}=1$ and $K$ is contained in an orthogonal group.

In characteristic $2$ it may happen that 
$\hat d_S^{K} =0$, and 
$d_S^{K}=1$. To exclude this possibility
in certain situations, it may be useful  the following:

\begin{lemma}\label{char2}  Let $g\in \GL_n(L)$, with $\char L = 2$.
Assume that $g$ is conjugate to its
inverse.  Then $d_S^g \ge n/2$ if
$n$ is even, $d_S^g \ge (n+1)/2$ if $n$ is odd.
\end{lemma}
\begin{proof}
Let $V=\left\{m\in \Mat_n(L)\mid gmg^T=m\right\}$. Since $g$ is conjugate to its inverse,
there exists $h\in \GL_n(L)$ such that $g ^T=h^{-1}g^{-1}h$.
It follows that $gmg^T=m$ if and only if $g$ centralizes $mh^{-1}$, whence
$\dim V =  d_M^g \ge n$. This inequality can be seen noting that, for each companion matrix $c$
of the rational form of $g$, the algebra $L[c]$ centralizes $c$.
Now consider the map from $V$ to $S$:
$m \mapsto m+m^T$.
The image of this map lies in $V \cap S$. Since  $\char L=2$,
the kernel is also in $V\cap S$. Since at least one of the dimensions
of the image and of the kernel is at least half of the dimension of $V$,
this completes the proof.
\end{proof}

\section{Groups preserving a form}\label{preserving}

The following theorem in the unitary case appeared in \cite[Lemma 6.2]{V2009}.
When $L$ is a finite field see also \cite[Theorem 2.12]{VZ}.

\begin{theorem}\label{forms}
Let $L$ be a field and  $\sigma\in \mathrm{Aut}(L)$. Let $X,Y\in
\mathrm{GL}_n(L)$. Suppose that $X^\sigma \sim X^{-1}$, $Y^\sigma
\sim Y^{-1}$, $(XY)^\sigma \sim (XY)^{-1}$. Assume further that
$\langle X, Y\rangle$ is absolutely irreducible and that \eqref{rigidity} holds.

{\rm (i)}  If $\sigma=\mathrm{id}$, then $\langle X, Y\rangle$
fixes a non-degenerate symmetric or skew-symmetric form.

{\rm(ii)}  If $\sigma$ is an involution, then $\langle X,Y\rangle$ 
fixes a non-degenerate hermitian form.
\end{theorem}

\begin{proof}
Let $M^\phi=\mathrm{Mat}_n(L)$ be the $\langle X,Y\rangle$-module
equipped with the following action $\phi$ on it:
 $$
 \phi(h).m =h^\sigma m h^T,
 $$
for all $h\in \langle X, Y\rangle$. Using the non-degenerate pairing $(m_1,m_2)=\tr{m_1m_2}$
we can identify $M^\phi$ with his dual. Via this identification, the dual representation $\phi^\ast$ is equivalent to
$\hat\phi$:
 $$
 \hat\phi(h).m =(h^T)^{-1} m (h^\sigma)^{-1}.
 $$

Clearly $h^\sigma$ is conjugate to its
transpose. If $h^\sigma\sim h^{-1}$, then choose an invertible matrix
$g$ such that $(h^\sigma)^T=g h^{-1} g^{-1}$. Then $h^\sigma m h^T=m$
if and only if $h(m^Tg)h^{-1}=m^Tg$. In particular,
$d_{M^\phi}^h=d_M^h$. Now, Scott's formula for the module $M^\phi$
together with  the assumptions of the Theorem and \eqref{rigidity} imply that either
$d_{M^\phi}^{\langle X, Y\rangle}\ge 1$ or $\hat{d}_{M^\phi}^{\langle
X, Y\rangle}\ge 1$.

First, we show that $d_{M^\phi}^{\langle X, Y\rangle}\ge 1$ yields
$\hat{d}_{M^\phi}^{\langle X, Y\rangle}\ge 1$. To this purpose assume that, for some
$m\neq 0$,
 \begin{equation}
 \label{eq:hmht=m}
 h^\sigma m h^T=m
 \end{equation}
for every  $h\in \langle X, Y\rangle$. Let $V$ be the
eigenspace of $m$ relative to 0. In particular $V\neq L^n$,
as $m\neq 0$. For any $h\in \langle X, Y\rangle$ and any $v\in V$, we
have $mh^Tv=(h^\sigma)^{-1} mv=0$. Therefore, $V$ is $\langle X^T,
Y^T\rangle$-invariant. By the absolute irreducibility 
of $\langle X,Y\rangle$ it follows $V=\{0\}$,  i.e., $m$ is non-degenerate. 
Inverting both
sides of (\ref{eq:hmht=m}), we have
$(h^T)^{-1}m^{-1}(h^\sigma)^{-1}=m^{-1}$. Hence,
$h^Tm^{-1}h^\sigma=m^{-1}$ and, since $\hat\phi$ is equivalent to
the dual representation $\phi^\ast$, we have  $\hat{d}_{M^\phi}^{\langle X,Y\rangle}\ge 1$, as desired.

Now consider the case $\hat{d}_{M^\phi}^{\langle X, Y\rangle}\ge 1$.
Assume that  $m\neq 0$ is such that
the equality
 \begin{equation}
 \label{eq:hmh=m-2}
 (h^T)^{-1} m (h^\sigma)^{-1}=m
 \end{equation}
holds for any $h\in \langle X, Y\rangle$. 
We show that $m$ is invertible. Let $V$ be the eigenspace of $m^T$
relative to 0. For any $h\in \langle X, Y\rangle$ and any $v\in V$,
we have $(hv)^Tm=v^Tm(h^\sigma)^{-1}=0$, whence $m^T(hv)=0$. Thus $V$ is $\langle X,
Y\rangle$-invariant and $V=\{0\}$ by the absolute irreducibility of
$\langle X, Y\rangle$. Therefore, $m$ is non-degenerate. In
particular, this implies that any two non-zero matrices in
$\mathrm{Mat}_n(L)$ satisfying (\ref{eq:hmh=m-2}) must be
proportional.

Equation (\ref{eq:hmh=m-2}) shows that $\langle X, Y\rangle$ fixes a
bilinear form $m$ defined over $L$. 

Transpose both sides of (\ref{eq:hmh=m-2}) and apply $\sigma$. We
have
$$
  (h^T)^{-1} (m^T)^\sigma (h^\sigma)^{-1}=(m^T)^\sigma.
$$
By what observed above,
\begin{equation}
\label{eq:beta}
 (m^T)^\sigma=\beta m \textrm{ for some } \beta\in L^\ast.
\end{equation}

(i) Assume that $\sigma=\mathrm{id}$. Repeating (\ref{eq:beta})
twice, we have $\beta=\pm1$, i.e., $m$ is either symmetric or
skew-symmetric.

(ii) Assume that $\sigma$ is an involution. Let $F={\rm Inv}(\sigma)$ be the 
subfield fixed pointwise by $\sigma$.  Our next aim is to find a
suitable scalar $\alpha\in L^\ast$ such that $\alpha m$ is hermitian,
i.e., $((\alpha m)^\sigma)^T = \alpha m$.  Iterating (\ref{eq:beta})
we have that $\beta\beta^\sigma=1$. By Hilbert's Theorem 90
 \cite[page 297]{J} for the extension $L/F$, there is $\alpha$ such that
$\beta=\alpha/\alpha^\sigma$. Therefore,
 $$
  ((\alpha m)^\sigma)^T=\alpha^\sigma \beta m =\alpha m,
 $$
as desired.
\end{proof}

\begin{corollary}\label{conformal}
Let $L$ be a field and  $\sigma\in \mathrm{Aut}(L)$. Let $X,Y\in
\mathrm{GL}_n(L)$. Suppose that for some $\lambda$, $\mu\in L$ we
have $X^\sigma \sim \lambda\lambda^\sigma X^{-1}$, $Y^\sigma \sim \mu\mu^\sigma Y^{-1}$,
$(XY)^\sigma \sim \lambda\mu(\lambda\mu)^\sigma (XY)^{-1}$. Assume further that
$\langle X, Y\rangle$ is absolutely irreducible and \eqref{rigidity} holds.

{\rm(i)} If $\sigma=\mathrm{id}$, then $\langle X, Y\rangle$ is
contained in a conformal orthogonal or in the conformal symplectic
group.

{\rm(ii)} If $\sigma$ is an involution, then $\langle X, Y\rangle$ is
contained in a conformal unitary group. Moreover, if $X$, $Y\in
\mathrm{GL}_n(F)$, where $F={\rm Inv}(\sigma)$ is the 
subfield fixed pointwise by $\sigma$,  then $\langle X, Y\rangle$ is contained in a conformal
orthogonal or in the conformal symplectic group defined over $F$.
\end{corollary}

\begin{proof}
Set $X_1=\lambda^{-1} X$, $Y_1=\mu^{-1} Y$. Then $X_1^\sigma
=(\lambda^{-1})^\sigma X^\sigma \sim \lambda X^{-1} = X_1^{-1}$,
$Y_1^\sigma \sim Y_1^{-1}$, $(X_1Y_1)^\sigma \sim (X_1Y_1)^{-1}$. By
Theorem~\ref{forms}, $\langle X_1, Y_1\rangle$ fixes a non-degenerate (symmetric or
skew-symmetric or, respectively, hermitian) form $m$. Hence
$$
 X^\sigma m X^ T= \lambda\lambda^\sigma m, \qquad
 Y^\sigma m Y^ T = \mu\mu^\sigma m,
$$
i.e., $\langle X, Y\rangle$ is contained in the corresponding
conformal group.

Moreover, the proof of Theorem~\ref{forms} shows that $m$ is unique
up to a scalar multiple. Therefore, if $\sigma$ is an involution and
$X$, $Y\in \mathrm{GL}_n(F)$, then $m=\alpha m_1$, where $m_1\in
\mathrm{GL}_n(F)$. (Since $\lambda\lambda^\sigma$, $\mu\mu^\sigma\in
F$, one can find the entries of $m_1$ as a solution of a system of
linear equations defined over $F$). Since $m$ is hermitian, we have
$m_1^T =(\alpha^\sigma/\alpha) m_1 =\beta m_1$. Applying $\sigma$ to the last
relation, we find $\beta=\beta^{-1}$, i.e., $m_1$ is either symmetric
or skew-symmetric.
\end{proof}

When  $\sigma=\mathrm{id}$, Theorem \ref{forms} does not allow to distinguish
between symmetric and skew-symmetric forms, as both cases may arise. We give some
conditions under which the symmetric case can be detected.

\begin{lemma}\label{negative}
Let $X,Y\in \GL_n(L)$ and suppose that $\left\langle X,Y\right\rangle$ is absolutely irreducible.
Assume further that $d^X_S+d^Y_S=\frac{n(n+1)}{2}$ and that $XY$ is conjugate to its inverse.
Then $\left\langle X,Y\right\rangle$ is contained
in an orthogonal group.
\end{lemma}

\begin{proof}

Setting $K=\langle X,Y\rangle$, relation  \eqref{two} gives $d^{XY}_S\leq d_S^K+\hat d_S^K$.
From the assumption that $XY$ is conjugate to its inverse it follows $1\leq d^{XY}_S$.
Actually, when $p=2$, we have the stronger condition $2\leq d^{XY}_S$ by Lemma \ref{char2}.
Now we make repeated use of
Lemma 1 of \cite{TV} which says, first, that $d_S^K\leq 1$ and $\hat d_S^K \leq 1$.
Moreover it says that, when $p$ is odd, $d_S^K=\hat d_S^K $.
We conclude  $d_S^K=\hat d_S^K =1$. Our claim
follows again from Lemma 1 of \cite{TV}.
\end{proof}

The previous Lemma is a special case of a more general fact (see Corollary \ref{cor to 3.1} below), 
which uses the
following result, essentially proved in \cite[Lemma 3.4]{V}.
\begin{lemma}
 \label{lem:fixform}
Let $g\in \mathrm{GL}_n(L)$ and let
 \begin{equation}
  \label{eq:fixform}
   g^Tmg=m
 \end{equation}
for some non-degenerate matrix $m$ which is either symmetric or skew-symmetric. Let $\mu_g$ be the
minimal polynomial of $g$. Assume that either {\rm (i)} $\deg
\mu_g>2a$ or {\rm(ii)} $\deg \mu_g=2a$ and the middle coefficient of
$\mu_g$ is non-zero. Then $d_S^g\ge a$.
\end{lemma}

\begin{proof}
Define $\theta:\mathrm{Mat}_n(L)\to S$ as follows:
$\theta(u)=u+u^T$. Clearly, if
 \begin{equation}
  \label{eq:fixform2}
  g^Tug=u,
 \end{equation}
then $g^Tu^Tg=u^T$ and $g^T\theta(u)g=\theta(u)$. Notice that for any
$i$ the matrix $u=mg^i$ satisfies (\ref{eq:fixform2}). Let
 $$
 U=\left\{ \sum_{i=1}^a c_i mg^i : c_i \in L\right\}.
 $$
It follows from (\ref{eq:fixform})  that $m^{-1} (g^i)^T m =g^{-i}$.
Set $\lambda=1$ if $m$ is symmetric,  $\lambda=-1$ if $m$ is skew-symmetric.
Consider
  \begin{eqnarray*}
  g^a m^{-1} \theta \left(\sum_{i=1}^a c_i mg^i\right) &=&
   g^a \sum_{i=1}^a c_i g^i  + g^a m^{-1} \sum_{i=1}^a c_i (g^i)^T
   m^T \\
  & =&
   g^a \sum_{i=1}^a c_i g^i   +\lambda g^a  \sum_{i=1}^a c_i m^{-1} (g^i)^T m
    \\
  &=&
    g^a \sum_{i=1}^a c_i g^i   +\lambda g^a  \sum_{i=1}^a c_i g^{-i} \\
  &=&
    \lambda \sum_{i=0}^{a-1} c_{a-i} g^i  +\sum_{i=a+1}^{2a} c_{i-a} g^i .
  \end{eqnarray*}
Clearly, under the assumptions of the Lemma, this sum is zero only if
all $c_i$ vanish. Therefore, the kernel of the restriction of
$\theta$ to $U$ is trivial and $\dim \theta(U)=\dim U=a$. In
particular, $d_S^g\ge a$.
\end{proof}

\begin{corollary}\label{cor to 3.1}
Let $X,Y\in \GL_n(L)$ be as in Theorem \ref{forms} with $\sigma= \id$ and assume further that:
\begin{equation}\label{ass}
d_S^X+ d_S^Y \geq \frac{n^2+n}{2}-\frac{\deg\mu_{XY}}{2}+2.
\end{equation}
Then $\left\langle X,Y\right\rangle$ is contained in an orthogonal group.
\end{corollary}
\begin{proof}
By Theorem \ref{forms}(i), there exists a non-degenerate, symmetric or skew-symmetric,
form $m$  which is preserved by $X$ and $Y$.
Clearly  $\deg \mu_g> 2\left(\frac {\deg\mu_g}{2}-1\right)$.
Thus, applying Lemma \ref{lem:fixform} to $g=XY$, we have $d_S^{XY}\geq \frac{1}{2}\deg\mu_{XY}-1$.
If $L$ has characteristic $2$, we have the stronger inequality 
$d_S^{XY}\geq \frac{n}{2}\geq \frac{1}{2}\deg \mu_{XY}$ 
by Lemma \ref{char2}.
Thus, under assumption \eqref{ass} we get
$$d_S^X+d_S^Y+ d_S^{XY}\geq \frac{n^2+n}{2}+1$$
and, when  $L$ has characteristic $2$
$$d_S^X+d_S^Y+ d_S^{XY}\geq \frac{n^2+n}{2}+2.$$
Our claim follows by the same arguments used in
the proof of Lemma \ref{negative}.
\end{proof}

\section{Canonical forms and shapes of a $(2,k)$-generating pair}\label{shape}

Let $\F$ be an algebraically closed field of characteristic $p\geq 0$.
We fix an integer $k\geq 3$. If $p=0$ or $(p,k)=1$, we
denote by $\epsilon$ a primitive $k$-th root of unity in $\F$.
If $(p,k)=p$ we suppose $k\in \left\{p,2p\right\}$
and, for $k=p\geq 3$, we set $\epsilon=1$,
for $k=2p\geq 4$ we set $\epsilon =-1$.

Next we consider two linear transformations  $\xi,\eta$ of $\F^4$ with
respective similarity invariants:

\smallskip
$\bullet$ $t^2-d,\ t^2-d$,\enskip  $d=\pm 1$,

\smallskip
$\bullet$
$t-1,\ t^3-(1+s)t^2+(1+s)t-1$,\enskip $s=\epsilon +\epsilon^{-1}$.

\smallskip
The projective image of $\xi$ has order 2. The Jordan form of $\eta$ is respectively
\begin{equation}\label{Jordan}
\left(
\begin{array}{cccc}
1&&&\\
&1&&\\
&&\epsilon&\\
&&&\epsilon^{-1}
\end{array}\right), \quad
\left(\begin{array}{cccc}
1&&&\\
&1&0&0\\
&1&1&0\\
&0&1&1
\end{array}\right), \quad
\left(\begin{array}{cccc}
1&&&\\
&1&&\\
&&-1&0\\
&&1&-1
\end{array}\right)
\end{equation}
according as (i) $(k,p)=1$, (ii) $k=p$ or $k=4$ and $p=2$, (iii) $k=2p\geq 6$.
It follows that
$\eta$ and its projective image have order $k$.

We assume further that $\left\langle \xi,\eta \right\rangle$
acts irreducibly on $\F^4$. As the eigenspace $V$ of $\eta$ relative to 1
has dimension 2,  and  $V\cap \xi (V)=0$ by the irreducibility of $\left\langle \xi,\eta \right\rangle$,
we have $V+\xi (V)=\F^4$.
Thus ${\cal B}=\left\{ \xi(v_1), \xi(v_2), v_1, v_2\right\}$ is basis of $\F^4$ whenever
$\left\{v_1,v_2\right\}$ is a basis of $\xi(V)$.
Considering the rational canonical form of the linear transformation
induced by $\eta$ on $\F^4/V$,
we may assume that $v_1$ and $v_2$ are chosen so that
the matrices of  $\xi$ and $\eta$, with respect to ${\cal B}$, have shapes:
\begin{equation}\label{generators}
x=\left(
\begin{array}{cccc}
0&0&1&0\\
0&0&0&1\\
d&0&0&0\\
0&d&0&0
\end{array}\right), \quad
y=\left(
\begin{array}{cccc}
1&0&r_1&r_2\\
0&1&r_3&r_4\\
0&0&0&-1\\
0&0&1&s
\end{array}\right)
\end{equation}
for suitable $r_i\in \F$ with $\left(r_1,r_3\right)\neq \left(r_2,r_4\right)$ 
if $k=p$ or $k=4$ and $p=2$ (i.e. if $s=2$).

For $\delta =\pm \sqrt d$, the corresponding eigenspace of $x$ is:
\begin{equation}\label{eigenspacex}
\left\{(a,b,\delta a, \delta b)^T\mid a,b\in \F \right\}.
\end{equation}

For $\epsilon \neq 1$, the eigenspace of $y$ relative to $\epsilon^j$ ($j=\pm 1$)  is generated by:
\begin{equation}\label{eigenvectory}
u_{\epsilon^{j}}\ =\ (r_1-\epsilon^{j}r_2,\ r_3-\epsilon^{j}r_4,\ \epsilon^{j} -1,\ -\epsilon^{2j}+\epsilon^{j})^T.
\end{equation}
Clearly when $k=2p\geq 6$, i.e., $\epsilon= -1$,  the two vectors coincide.


\begin{lemma}\label{reducible} Let  $x,y$ be defined as in \eqref{generators}.
Then $H=\left\langle x,y\right\rangle$ is a reducible subgroup of $\SL_4(\F)$ if and
only if one the following conditions holds for some $j=\pm 1$, and some $\delta=\pm \sqrt d$:

\smallskip
{\rm (i)}\enskip $r_4=r_1-\epsilon^{j} r_2+\epsilon^{-j} r_3$;

\smallskip
{\rm (ii)}\enskip $\Delta = r_1r_4 - r_2r_3+\delta^{-1}\left((s -1) r_1 - r_2 + r_3- r_4\right) + (2-s)d=0$.
\end{lemma}

\begin{proof}

If (i) holds, taking $v=\left(-\epsilon^{-j},1,0,0\right)^T$,
we get $yxv= d(r_1-\epsilon^jr_2)v+\epsilon^jxv.$
Thus $\left\langle v, xv\right\rangle$
is a 2-dimensional $H$-module. On the other hand, if (ii) holds, there exists
$(a,b)\neq (0,0)$ such that
$w=(\delta a, \delta b ,a,b)^T$ is fixed by
$y^T$. Since $w$ is an eigenvector of $x^T$, the 1-dimensional space $\left\langle w\right\rangle$ is $H^T$-invariant.
It follows that $H$ has a 3-dimensional submodule.

\smallskip
Viceversa,  let $W$ be a proper $H$-submodule.

\smallskip
{\bf Case 1.} dim $W=1$. In this case $W$ is generated by a common eigenvector $u$ of $x$ and $y$.
From $xW=W$ we get $u\not\in \left\langle e_1,e_2\right\rangle$. Hence
$k\neq p$ if $p$ is odd, $k\ne 4$ if $p=2$ and, up to a scalar, $u=u_{\epsilon^{j}}$ for some $j=\pm 1$.
From $xu=\delta u$ for some $\delta=\pm \sqrt{d}$, we get
$$r_1=\epsilon^{j}r_2+\delta^{-1}(\epsilon^{j} -1),\quad
r_3=\epsilon^{j} r_4+\delta^{-1}(\epsilon^{j}-\epsilon^{2j}).$$
These conditions imply condition (i).

\smallskip
{\bf Case 2.} dim $W=3$.
In this case $H^T$ has a 1-dimensional invariant space.
A generator must have shape $(\delta a, \delta b ,a,b)^T$, in order to be an eigenvector of $x^T$.
And a non-zero vector $w$ of this shape is an eigenvector of $y^T$
only if $y^Tw=w$. This condition gives (ii).

\smallskip
{\bf Case 3.}  dim $W=2$. Assume first that there is a non-zero $v\in W$ such that $yv=v$.
It follows that $v$ and $xv$ are linearly independent, hence generate $W$.
By the shape of $y$, we may assume $v=(\alpha, 1,0,0)^T$ for some $\alpha\in \F$. From
$yxv=\lambda v+\mu xv$ we get $\mu\alpha=-1$, $\mu =\epsilon^j$, and these conditions
easily give that (i) must hold.
Next suppose that $yv\ne v$ for all non-zero $v\in W$.
It follows that $\epsilon \ne 1$ and the characteristic polynomial of the linear transformation
$\eta_0$, say, induced by $y$ on $\F^4/W$ must be $(t-1)^2$. Considering
the minimum polynomial of $y$, we have 
$(\eta_0-I)(\eta_0-\epsilon I)(\eta_0-\epsilon^{-1} I)=0$.
As the second and third factors are invertible, we get $\eta_0=I$.
Thus $y$ must induce the identity on $\F^4/W$. So we get that $H^T$
fixes the  space $U$ of the fixed points of $y^T$. Since $U$ is fixed by $x^T$,
we can choose a non-zero vector $w\in U$ which is an eigenvector
of $x^T$. Hence $\langle w\rangle$ is  $H^T$-invariant,
and condition (ii) must hold.
\end{proof}

The characteristic polynomials of $xy$ and $(xy)^{-1}$ are respectively:

\begin{eqnarray}
   \chi_{xy}(t) &=& t^4-d(r_1+r_4)t^3+
              (r_1r_4-r_2r_3-ds)t^2 +(r_1s-r_2+r_3)t+1 ;
             \label{general car xy}
\\
 \chi_{(xy)^{-1}}(t) &=& t^4+(r_1s-r_2+r_3)t^3+
              (r_1r_4-r_2r_3-ds) t^2 -d(r_1+r_4)t+1.
             \label{general car xy^-1}
\end{eqnarray}

\begin{remark}\label{rigid triple} If $x$,$y$ are as in \eqref{generators},
by a formula of Frobenius \cite [Theorem 3.16, p. 207]{J},
\begin{equation}\label{centralizers}
\dim (C_{\Mat_4(\F)}(x))=8,\quad \dim (C_{\Mat_4(\F)}(y))=6.
\end{equation}
When $H=\left\langle x,y \right\rangle$ is absolutely irreducible, from \eqref{centralizers} and
\eqref{one} we get $\dim (C_{\Mat_4(\F)}(xy))=4$.
In particular equality holds in \eqref{one}, i.e., the triple $(x,y,xy)$ is rigid.
Moreover $xy$ has a unique similarity invariant, equivalently  its minimal and characteristic 
polynomials coincide.
\end{remark}

\section{Field of definition}\label{sec: field definition}
For lack of a reference,  we sketch a proof of the following  well known fact.

\begin{lemma} Let $\Omega$ be the set of coefficients
of the similarity invariants of $h\in \GL_n(\F)$.
If $h^g\in \GL_n(\F_1)$, with $g\in \GL_n(\F)$
and $\F_1\leq \F$, then $\Omega\subseteq \F_1$.
\end{lemma}
\begin{proof}
Call $C$ and $C_1$ the rational canonical forms of
$h$ and $h^g$ respectively in $\GL_n(\F)$ and $\GL_n(\F_1)$.
From $C_1$ conjugate to $C$ in $\GL_n(\F)$, we have $C_1=C$.
As all elements of $\Omega$ appear as entries of $C$, we conclude that $\Omega\subseteq F_1$.
\end{proof}

We denote the centre of $\GL_4(\F)$ by $\F^\ast I$.

\begin{lemma}\label{minimal field} For $x,y$ defined as in \eqref{generators},
let $H=\left\langle x,y\right\rangle$ be conjugate to a subgroup of $\GL_4(\F_1)\F^\ast I$,
for some $\F_1\leq \F$. Then $\F_1\geq \F_p\!\left[s,(r_1+r_4)^2\right]$, where $\F_p$ denotes the
prime subfield.
\end{lemma}
\begin{proof}
Assume $H^g\leq \GL_4(\F_1)\F^\ast I$ and write $x^g=x_1\lambda^{-1}$, $y^g=y_1\rho^{-1}$
with $x_1,y_1\in \GL_4(\F_1)$, $\lambda ,\rho\in \F^\ast$.
The similarity invariants of
$\rho y$ are $t-\rho$, $t^3 - \rho(s+1)t^2+ \rho^2(s+1)t -\rho^3$.
As $(\rho y)^g=\rho y^g=y_1$, it follows from the previous Lemma that $\rho$ and $s$ are in $\F_1$.
In the same way, as the similarity invariants of $\lambda x$ are $t^2-d\lambda^2$,
$t^2-d\lambda^2$, $d=\pm 1$; we get $\lambda^2 \in \F_1$. As $x_1y_1=(\lambda x\rho y)^g$
has trace $d\lambda\rho(r_1+r_4)\in \F_1$,
it follows that $(r_1+r_4)^2\in \F_{1}$.
\end{proof}

\begin{lemma}\label{subfields} Let $q=p^a$ be a prime power, with $a>1$. 
Denote by $N$ be the number of non-zero elements $r\in \F_q$ such that
$\F_p\!\left[r^2\right]\ne \F_q$. Then:

$$\left.
\begin{array}{ll}
N\le  2(p-1) & \text{\ if\ } a=2;\\
N\le (p-1,2)\frac{p \left(p^{\left\lfloor a/2\right\rfloor}-1\right)}{p-1}& \text{\ if\ }  a>2.
\end{array}\right .$$
\end{lemma}
\begin{proof}

For each  $\alpha\in \F_q^\ast$ such that
$\F_p\!\left[\alpha\right] \ne \F_q$, there are at most $2$
elements $r\in \F_q$ such that $r^2=\alpha$. And, if $p=2$
there is just $1$.
Thus our claim is clear for $a=2$.
For $a>2$, considering the possible orders of subfields of $\F_q$, we have
$$N\le (p-1,2) \left(p+\dots +p^{\left\lfloor a/2\right\rfloor}\right)\le  (p-1,2) \frac{p\left(p^{\lfloor a/2\rfloor}-1\right)}{p-1}.$$

\end{proof}

\section{Negative results}\label{negres}

Let $X,Y$ be elements of $\SL_4(\F)$ whose projective 
images have orders $2$ and $3$.
Then $X^2=i^hI$, for some $h=0,1,2,3$, where $i$
satisfies $i^2+1=0$. Multiplying $Y$ for a scalar, if necessary,
we may suppose that $Y^3=I$. Both $X$ and $Y$ have at least 2 
similarity invariants, whence $d^X_M\geq 8$ and $d^Y_M\geq 6$. Now assume that
$\left\langle X,Y\right\rangle$ is irreducible.
If $X$ or $Y$ has more than 2 similarity invariants, we get 
$d^X_M\ge 10$ or $d^Y_M\ge 10$. From  $d^{XY}_M\ge 4$,
we get a contradiction with respect to \eqref{one}. It follows that $X$ and $Y$ 
have 2 similarity invariants, which necessarily coincide with those
of $x$  and $y$ in \eqref{generators}, with $s=-1$ and some $d=\pm 1$.

\begin{theorem}\label{nr} Let $q=p^a$ be a prime power.

{\rm (i)} $\Sp_4(q)$ is not $(2,3)$-generated.

{\rm (ii)} If $p=2,3$, then $\PSp_4(q)$ is not $(2,3)$-generated.

{\rm (iii)} $\SL_4(2)$ is not $(2,3)$-generated.

{\rm (iv)} $\Sp_4(q)$ is not generated by elements having
the same similarity invariants as $x$ and $y$.
\end{theorem}

\begin{proof}
(i) and (ii). Let $X,Y\in \Sp_4(q)$ be preimages of a
$(2,3)$-generating pair of $\PSp_4(q)$.
By the above considerations,
we may assume $X=x$ and $Y=y$ as in \eqref{generators}, with $d=\pm 1$, $s=-1$.

If $d=1$, then  $d^x_S=6$. Moreover $d^y_S=4$. 
Noting that $xy$ is conjugate to its inverse, being a symplectic matrix, 
we have that $\langle x,y \rangle$ is contained in an orthogonal group,
by  Lemma \ref{negative}.
This contradiction proves (i) and also (ii) when $p=2$.

If $d=-1$ and $p=3$, equating $\tr{xy}$ and $\tr{(xy)^{-1}}$
we get $r_4=r_1-r_2+r_3$. As $\epsilon =1$, the group $\left\langle x,y\right\rangle$ is reducible
by Lemma \ref{reducible}(i): a contradiction.

(iii)  Let $X,Y$ be a $(2,3)$ pair in $\SL_4(2)$,
which generates an absolutely irreducible subgroup.
Up to conjugation $X=x$, $Y=y$ as in \eqref{generators}, with $s=-1$.
In all the cases in which $\langle x,y \rangle$ is irreducible, namely
$\left(r_1,r_2,r_3,r_4\right)\in \left\{(1,0,0,0), (0,1,0,1),(0,0,1,1)\right\}$,
we have $r_4=r_2+r_3$. It follows that $xy$ and $(xy)^{-1}$ have the same characteristic polynomial.
Hence, by Remark \ref{rigid triple}, they are conjugate. As above, case $d=1$,
$\langle x,y\rangle$ is contained in an orthogonal group.

(iv) If the claim is false, then $\Sp_4(q)$ could be generated by $x,y$ of shape \eqref{generators}, with $d=1$. 
Again, by  Lemma \ref{negative}, the group $\left\langle x,y\right\rangle$ is contained in an orthogonal group.
A contradiction.
\end{proof}

\bigskip
The first three points of the previous Theorem give a unified  proof of known results.
Indeed it had been shown by Liebeck and Shalev \cite[Proposition 6.2] {LS} that  $\PSp_4(q)$ is not $(2,3)$-generated
for $p=2,3$. And $SL_4(2)\cong \Alt(8)$ is not $(2,3)$ generated by a result of Miller
\cite{ML}.

We recall the presentations of certain groups that will be used. Some of them are well known.

\begin{lemma}\label{pres} Let $G$  be a non trivial group. In {\rm (i)}-{\rm (iv)}
suppose that $G=\langle S,T\rangle$.

{\rm (i)} If $S^2=T^3= (ST)^5=1$ then $G\cong \Alt(5)$;

{\rm (ii)} if $S^2=T^3=(ST)^7=[S,T]^4=1$,  then $G\cong \PSL_2(7)$;

{\rm (iii)} if $S^2=T^7=(TS)^3=(T^4S)^4=1$,  then $G\cong \PSL_2(7)$;


{\rm (iv)} if $S^2=T^4=(ST)^7=
(ST^2)^5=(T(ST)^3)^7
    =T(ST)^3T^2(ST)^3
   S(T(T(ST)^3)^2TST)^2=1$  
   
   then $G\cong \PSL_3(4)$;

{\rm (v)} If $G=\langle T_1,T_2,T_3\rangle$ and
$T_1^2=T_2^2=T_3^2=(T_1T_2)^3=(T_2 T_3)^3=(T_1 T_3)^4=(T_1 T_2 T_3)^5=1$, then $G\cong \Alt(6)$;

\smallskip
{\rm (vi)} If $G=\langle T_1,T_2,T_3,T_4,T_5\rangle$, where $T_1,\ldots,T_5$ satisfy the following conditions

 $$\left\{ \begin{array}{ll}
T_i^3=1, & i=1,\ldots,5, \\
(T_i T_{i+1})^2=1, & i=1,\ldots,4 \\

[ T_i,T_j ]= 1 & |i-j|>2,\\
 T_i T_{i+1}^{-1} T_{i+2} T_i^{-1} T_{i+2}^{-1}=1,& i=1,2,3.
 \end{array}\right. ,$$

then  $G\cong \Alt(7)$.

\end{lemma}
\begin{proof}
See \cite{CM} and \cite{A} for (i)-(iii), 
\cite{CMY} for (iv) and  \cite[Theorem ~1]{VeVs} for (v), (vi). 
\end{proof}

\begin{lemma}
$\SU_4(4)$ and $\PSU_4(9)$ are not $(2,3)$-generated.
\end{lemma}

\begin{proof}
$\SU_4(4)\cong \PSp_4(3)$ is not $(2,3)$-generated by Theorem \ref{nr}(ii). 
By contradiction, let $X,Y$ be the preimage in $\SU_4(9)$ of a 
$(2,3)$-generating pair of $\PSU_4(9)$. By what observed at the beginning of 
this Section, we may assume that $X,Y$ are as $x,y$ in \eqref{generators}, 
for some $d=\pm 1$ and $s=-1$. Since $(xy)^3$ must be conjugate
to $(xy)^{-1}$, we obtain the conditions $r_1^3+r_4^3=d(r_1+r_2-r_3)$ and $r_1r_4-r_2r_3-2 \in \F_3$.
By Lemma \ref{reducible}, we have also to impose that 
$r_1-r_2+r_3\neq r_4$ and that $r_1r_4-r_2r_3\pm i(r_1-r_2+r_3-r_4)\neq 0$.
Finally, not all the $r_i$'s belong to the prime field. 
We list the possible $4$-tuples satisfying all these conditions, 
denoting by $\xi$ an element of $\F_9$ such that $\xi^2-\xi-1=0$.

\smallskip

{\bf Case $d=1$}. There are 48 such $4$-tuples. Namely:
$$A)\left. \begin{array}{ccc}
\pm (1,\xi,\xi^7,1) & \pm(\xi,1,\xi^7,\xi^6) & \pm(\xi,\xi^3,-1,\xi^6)
\end{array} \right. $$
$$B)\left\{\begin{array}{lllll}
\pm (0,0,\xi,\xi^7) & \pm (0,\xi,0,\xi^3) & 
\pm (\xi^2,\xi^5,\xi,-1) & \pm(\xi^2,\xi^6,\xi^5,\xi) & \pm(\xi^2,\xi,\xi^2,\xi)\\ 
\pm (\xi,0,\xi^6,0) & \pm (\xi,\xi^2,0,0) &  \pm (\xi,\xi^6,\xi^7,-1)& \pm  (\xi,\xi^3,\xi^2,-1)  
\end{array} \right. $$
and their images under the field automorphism $\xi\mapsto \xi^3$.

\smallskip
{\bf Case $d=-1$}. There are $54$ such $4$-tuples. Namely:
$$A)\left. \begin{array}{cccccc}
\pm (0,0,\xi^2,-\xi^2) & \pm(0,\xi^2,0,\xi^2) & \pm(\xi^2,0,0,0) & \pm(\xi,\xi^7,\xi^2,1) & \pm(\xi,\xi^6,\xi^3,1) & \pm (\xi^2,\xi^5,\xi^7,\xi^2)
\end{array} \right. $$
$$B)\left\{ \begin{array}{lllll}
\pm (0,0,\xi,\xi^3) & 
 \pm (\xi,0,1,0)  & \pm(\xi,1,\xi^3,\xi^2)  & \pm(\xi,\xi^7,-1,\xi^2) & \pm (1,-1,\xi,\xi^5)\\ 
\pm (0,\xi,0,\xi^7) & \pm (\xi,-1,0,0) &  \pm (1,\xi,\xi^5,\xi^6)& \pm  (1,\xi^5,1,\xi^5)  
\end{array} \right. $$
and their images under the field automorphism $\xi\mapsto \xi^3$.

In both cases $A)$ the matrix  $(xy)^5$ is scalar, hence $H/Z(H)\cong \Alt(5)$ by Lemma \ref{pres}(i).

In both cases $B)$, both $(xy)^7$ and $[x,y]^4$ are scalar and so $H/Z(H)\cong \PSL_2(7)$ by Lemma \ref{pres}(ii).
\end{proof}

\begin{lemma}\label{SL}
$\SL_4(3)$ and $\SU_4(9)$ are not $(2,4)$-generated.
\end{lemma}
\begin{proof} Let $G$ be one of the groups $\SL_4(3)$, $\SU_4(9)$ and assume by contradiction that
$X,Y$ is a $(2,4)$-generating pair of $G$. Clearly $X$ must have Jordan form
$\diag(1,1,-1,-1)$ and $Y$ must have Jordan form either
$\diag(i,i,-i,-i)$ or $i^h\cdot \diag(1,-1,i,i)$, for some $h=0,1,2,3$.
The first possibility is excluded by \eqref{one}.
For the remaining possibilities we make the following observations.
If $G=\SL_4(3)$, then $h=1,3$ in order that $\tr{Y} $ lies in $\F_3$.
From $\langle X,-Y\rangle \leq \langle X,Y,-I\rangle \leq G$, with $G$ perfect
we deduce that $\langle X,Y\rangle = G$ iff $\langle X,-Y\rangle = G$.
By similar considerations, when $G=\SU_4(9)$, we get
$\langle X,Y\rangle = G$ iff $\langle X,i^hY\rangle = G$, $h=0,1,2,3$.
So, up to conjugation,  we may suppose $X=x$, $Y=y$ as in \eqref{generators}, with $d=1$, $s=0$.

Assume that $r_1+r_2-r_3+r_4=0$.
Setting $J=\left(\begin{array}{cc}
 J_1&J_2\\
 -J_2&-J_1
 \end{array}\right)$ with
 $$J_1=\left(\begin{array}{cc}
 0&1\\
 -1&0\end{array}\right),\quad J_2=\left(\begin{array}{cc}
 -r_3-r_4&r_3-r_4\\
 r_3-r_4&r_1+r_3-r_4
 \end{array}\right)$$
 we get
 $x^TJx=-J$, $y^TJy=J$. As $J$ is non-zero, we have $\langle x,y\rangle \ne G$.

So, from now on, we suppose that
\begin{equation}\label{assump}
r_1+r_2-r_3+r_4\ne 0.
\end{equation}
By Lemma \ref{reducible}(ii), we must have:
\begin{equation}\label{condition ii}
 \Delta =r_1r_4-r_2r_3\pm (r_1+r_2-r_3+r_4)-1\ne 0.
 \end{equation}
We claim that
$r_1r_4-r_2r_3\in \F_3$ and $(r_1+r_2-r_3+r_4)\in \F_3$
also when $G=\SU_4(9)$. Indeed, in this case,
$(xy)^{-1}$ and $(xy)^\sigma$ have the same
characteristic polynomial. Comparing
\eqref{general car xy} with \eqref{general car xy^-1}
it follows that $r_1r_4-r_2r_3\in \F_3$ and $r_1=(r_2-r_3)^3-r_4$.
Hence also $r_1+r_2-r_3+r_4= (r_2-r_3)^3+(r_2-r_3)\in \F_3$.

\smallskip
{\bf Case} $r_1r_4-r_2r_3\in \{0,-1\}$.
Under assumption \eqref{assump}, we have $\Delta\ne 0$
precisely when $r_1+r_2-r_3+r_4\not\in\F_3$. So this case does not arise.

\smallskip
{\bf Case} $r_1r_4-r_2r_3=1$.
In this case  $\Delta\neq 0$ when $r_1+r_2-r_3+r_4=\rho$ with $\rho=\pm 1$.

The elements $(r_1,r_2,r_3,r_4)$ of $\F_3^4$ for which $\Delta\neq 0$ give rise to products $xy$
whose characteristic polynomial has shape:
$t^4+(r_2-r_3-\rho)t^3+t^2+(-r_2+r_3)t+1$.
If  $r_2-r_3=-\rho$, then
$\chi_{xy}(t)= \chi_{(xy)^{-1}}(t)$. It follows that $xy$ is conjugate to $(xy)^{-1}$,
whence $\langle x,y\rangle$ is contained in an orthogonal group by Lemma \ref{negative}.
If $r_2-r_3=\rho$, we get the solutions $r_1=r_4=0$, $r_2=-r_3=-\rho$.
And, if $r_2-r_3=0$, we get the solutions $r_1=r_4=-\rho$, $r_2=r_3=0$.
In both cases $\langle x,y\rangle$ is reducible by Lemma \ref{reducible}(i).

\smallskip
The elements $(r_1,r_2,r_3,r_4) $of $\F_9^4\setminus \F_3^4$ for which $\Delta\neq 0$
give rise to products $xy$ whose characteristic polynomial has shape:
$t^4+(r_3-r_2)^3t^3+t^2+(r_3-r_2) t+1$ with $(r_3-r_2)\in \left\{\pm 1,\pm \sqrt i, \pm (\sqrt i)^3\right\}$.
If $r_3-r_2=\rho=\pm 1$, then $(xy)^5=\rho I$. Setting $T_1=x$,  $T_2=\rho xy^{-1}xy^2xyx$ and $T_3=y^{-1}xy$, then their respective projective images satisfy the presentation given by Lemma \ref{pres}(v). Since $y=T_3(T_1 T_2T_3)^2 T_1 T_2 T_1$, we have that $H/Z(H)\cong \Alt(6)$.

So, up to  field automorphisms, we are left to consider one quadruple $(r_1,r_2,r_3,r_4)$
for each of the cases $(r_3-r_2)=\sqrt i$ and $(r_3-r_2)=-\sqrt i$, e.g. $(\sqrt i,\sqrt i,1,\sqrt i)$ and $(-\sqrt i,1,\sqrt i, -\sqrt i)$. Direct calculation shows that the projective images $\bar x,\bar y$ of $x,y$,
satisfy the relations of Lemma \ref{pres}(iv), which define $\PSL_3(4)$. Since conjugate rigid triples
generate conjugate subgroups, we have reached a contradiction.
\end{proof}

\begin{theorem}\label{PSU}
$\PSU_4(9)$ is not $(2,4)$-generated.
\end{theorem}

\begin{proof}
Assume that $(X,Y)$ is a preimage in $\SU_4(9)$ of a
$(2,4)$-generating pair of $\PSU_4(9)$.
Since $\PSU_4(9)$ has one class of involutions
and two classes of elements of order 4, in virtue of Lemma \ref{SL} we are left to
consider the case in which $X$ and $Y$ have respective
Jordan forms $(1,1,-1,-1)$ and $\left(\xi , \xi ^3, \xi^5 , \xi ^7\right)$,
with $\xi^2-\xi-1=0$.
Let $\{v_1,v_2,v_3,v_4\}$ be such that
 $$
   Yv_1=\xi v_1, \quad
   Yv_2=\xi^3 v_2, \quad
   Yv_3=\xi^5 v_3, \quad
   Yv_4=\xi^7 v_4.
 $$
Clearly the  vectors $v_i$ are defined up to non-zero scalar
multiples.
In particular the spaces $W=\left\langle v_1,v_2\right\rangle$ and
$\left\langle v_3,v_4\right\rangle$ must be totally isotropic.
Since $\mathrm{SU}_4(9)$ is absolutely irreducible, we have
$\dim \left(W \cap XW\right) <2$.

\medskip
{\bf Case 1}: $\dim \left( W\cap XW\right)=1$. Therefore, there is a non-zero vector $v\in W$
such that $X v = \mu v$, $\mu=\pm1$.
The Gram matrix of the hermitian form fixed by $Y$
with respect to the basis $\left\{v_1,v_2,v_3,v_4\right\}$ must have shape:
\begin{equation}\label{Gram}
\left(
 \begin{array}{cccc}
    0 & 0 & b_1 & 0 \\
    0 & 0 & 0   & b_2 \\
    b_1^\sigma & 0 & 0 & 0 \\
    0 & b_2^\sigma & 0 & 0
 \end{array}
 \right).
\end{equation}
Again by the
absolute irreducibility of $\mathrm{SU}_4(9)$, we have $\langle
v\rangle \neq \langle v_1\rangle$ and $\langle v\rangle \neq \langle
v_2\rangle$.
Replacing $v_1$ and $v_2$ by appropriate scalar multiples, if necessary,
we can always assume that $v=v_1+v_2$. Moreover, keeping this choice for $v_1$ and $v_2$,
it is possible to replace $v_3$ and $v_4$ by appropriate scalar multiples in order to obtain
$b_1=b_2=1$ in \eqref{Gram}.

Now let us consider the basis $\{v_1+v_2$, $v_1-v_2$, $v_3+v_4$, $v_3-v_4\}$ of $\mathbb{F}_9^4$.
The Gram matrix with respect to this basis remains the same, up to a scalar. Thus we have:
$$
X=\left(
 \begin{array}{cccc}
    a& c_1 & d_1 & f_1 \\
    0 & c_2& d_2   & f_2 \\
    0& c_3 & d_3& f_3 \\
    0 & c_4 & d_4 & f_4
 \end{array}
 \right),\ a=\pm 1\quad J=\left(
 \begin{array}{cccc}
    0 & 0 & 1 & 0 \\
    0 & 0& 0   &1 \\
    1& 0 & 0&0 \\
    0 & 1 & 0&0
 \end{array}
 \right).$$
Imposing $X^TJ=JX^\sigma$ we get $c_3=f_3=0$ and $d_3=a$. $\tr{X}=0$ gives $f_4=-2a-c_2=a-c_2$.
After this substitution, the entry $(4,2)$ of $X^2$ is $ac_4$. Thus $c_4=0$ and
$\left\langle v_1,v_2\right\rangle$ is $\left\langle X,Y\right\rangle$-invariant.
Therefore, case 1 actually does not hold.

\medskip
{\bf Case 2}: $\dim \left( W\cap XW\right)=0$.
Therefore $v_1, v_2, Xv_1, Xv_2$ is a basis for $\F_9^4$.
Interchanging $v_3$ with $v_4$ and replacing  them by appropriate
scalar multiples, if necessary, we may assume either that
$v_3=w+Xv_1$ or that $v_3=w+Xv_1+Xv_2$ for some $w\in W$.
With respect to this basis, $X=x$ and $Y$ has one of the following shapes:
$$y_1=\left(
\begin{array}{cccc}
\xi &0&r_1&r_2\\
0&\xi^3&r_3&r_4\\
0&0&a&b\\
0&0&0&-a-1
\end{array}\right),\quad
y_2=\left(
\begin{array}{cccc}
\xi &0&r_1&r_2\\
0&\xi^3&r_3&r_4\\
0&0&c&1-c-c^2\\
0&0&1&-c-1
\end{array}\right)
$$
where $a^2+a-1=0$, $b=0,1$, $c\in \F_9$.
The Gram matrix of the form with respect to this basis
has now shape $J=\left(
\begin{array}{cc}
0&B\\
B&0
\end{array}\right)$, with $B= \left(\begin{array}{cc}
b_1&c_1\\
c_1^3&b_2
\end{array}\right)$, where $b_1,b_2\in\mathbb{F}_3$.

\smallskip
{\bf Case 2.1:} $Y=y_1$ with $a=\xi^7$, $b=0$.

Then $b_1=b_2=0$ and, multiplying $e_1$ by a suitable constant, we may assume that $c_1=1$.
Then we have that $r_4=-{r_1}^3$, $r_2\in \{0,\xi^3,-\xi^3\}$,
$r_3\in \{0,\xi,-\xi\}$. Let
$$
 K= \left(\begin{array}{cccc}
  0   & 0  & 1   & 0 \\
  0   & 0  & 0   & \pm1 \\
 -1   & 0  & 0   & 0 \\
  0   & \mp1  & 0   & 0
\end{array}\right),
$$
By a direct computation $x^TKx=-K$. Moreover, $y_1^TKy=K$ if and only
if $r_2=\pm\xi^{-2}r_3$. Thus, if $r_2=\pm\xi^{-2}r_3$, then $\langle
x,y_1\rangle \le \mathrm{CSp}_4(9)$. Hence $\langle x,y_1\rangle\neq
\mathrm{SU}_4(9)$. In the remaining cases $(r_2,r_3)$ is one of the
pairs $(0,\xi)$, $(0,-\xi)$, $(\xi^3,0)$, $(-\xi^3,0)$. In
particular, $\langle x,y_1\rangle$ is reducible since  either
$\langle v_1,Xv_1\rangle$ or $\langle v_2,Xv_4\rangle$ is $\langle
x,y_1\rangle$-invariant.

\smallskip
{\bf Case 2.2:} $Y=y_1$ with $a=\xi^7$, $b=1$.

By direct computation we deduce that $b_1=0$ and $c_1^3 =
\xi^2 b_2$. If $b_2=0$, then $c_1=0$ and the form is 0. Since the
form is defined up to a scalar multiple, we can assume that that
$b_2=1$ and $c_1=\xi^6$.

The condition $y_1^T J y_1^\sigma=J$ implies that $r_3=a_1\xi^3$,
$a_1\in \mathbb{F}_3$,
$r_1=-r_4^3-r_3=
-r_4^3-a_1\xi^3$, and
$r_2=-r_4+a_2\xi$, $a_2\in \mathbb{F}_3$.

If, in addition, $r_4=a_3+a_1\xi^2$, where $a_3\in \mathbb{F}_3$,
then taking
$$
K=\left(
 \begin{array}{cccc}
  0             &  1     & -a_1  &  \xi a_1 + a_3 \\
 -1             &  0     & \xi a_1 + a_3  & -\xi^2 a_1 + a_2 - \xi^3a_3 \\
 a_1            &  -\xi a_1 -a_3 &  0&   -1 \\
 -\xi a_1 -a_3 &  \xi^2 a_1 - a_2 + \xi^3 a_3 &  1  & 0
 \end{array}
 \right)
$$
we obtain
$$
  x^T K x =-K, \qquad y_1^T K y_1 = -K.
$$
Therefore, in that case $\langle x, y_1\rangle$ is either reducible
(if $K$ is degenerate) or is contained in $\mathrm{CSp}_4(9)$ (if $K$
is non-degenerate).

Notice that $r_3\neq 0$, otherwise $\langle
(1,0,0,0)^T,(0,0,1,0)^T\rangle$ is $\langle x,y_1\rangle$-invariant.

Altogether, there are 36 possibilities left. The remaining cases are:
\begin{eqnarray*}
  r_3 = a_1 \xi^3, &\qquad& r_4=a_3+a_4\xi^2, \\
  r_1 =-a_3+a_4\xi^2-a_1\xi^3 &\qquad& r_2=-a_3-a_4\xi^2+a_2\xi,
\end{eqnarray*}
where $a_1,a_2,a_3,a_4\in \mathbb{F}_3$, $a_1\neq0$, $a_4\neq a_1$.

In the following analysis, we denoting by   $\bar{x}$,   $\bar{y}_1$ 
the projective images of $x,y_1$.

{\bf Case 2.2.1:} If $(a_1,a_2,a_3,a_4)$ is one of the following tuples
$$\left.
\begin{array}{cccc}
 \pm (1,0,0,0), & \pm(1,1,0,0), & \pm(1,1,-1,0), & \pm(1,-1,-1,0).\\
 \end{array}\right.
$$
then $\langle \bar{x},\bar{y}_1\rangle \cong \mathrm{PSL}_3(4)$
by Lemma \ref{pres}(iv). 

\smallskip
{\bf Case 2.2.2:} Set  $S=\bar{x}$,   $T=(\bar{y}_1\bar x)^2$
if $(a_1,a_2,a_3,a_4)$ is one of the following tuples
$$\left. \begin{array}{cccc}
\pm(1,0,-1,-1),& \pm(1,1,0,-1),& \pm(1,-1,0,-1),& \pm(1,-1,-1,-1).\\
\end{array}\right.$$
Then $\langle S,T\rangle\cong \mathrm{PSL}_2(7)$
by Lemma \ref{pres}(iii). Since $S$ has odd order,
we have $\langle \bar{x},\bar{y}_1\rangle =\langle T,S\rangle$.

\smallskip
{\bf Case 2.2.3:} If $(a_1,a_2,a_3,a_4)=\pm(1,0,1,0)$, then set $g=(x y_1)^4$. In
particular, $g$ is non-scalar and $g^2=1$.  Since $xgx^{-1}=-g$,
$y_1gy_1^{-1}=-g$ and $-1=y_1^4\in \langle x,y_1\rangle$, we have
that $\langle x,y_1\rangle$ possess a non-central normal subgroup of
size 4. Hence, $\langle \bar{x},\bar{y}_1 \rangle\ne
\mathrm{PSU}_4(9)$. 

Actually a MAGMA calculation gives that $\langle x,y_1\rangle$ is a
2-group of size 128.

\smallskip
{\bf Case 2.2.4:} If $(a_1,a_2,a_3,a_4)=\pm(1,0,0,-1)$ or $\pm(1,1,-1,-1)$, take $g_1,\ldots,g_4$ such that:
$$g_1=(xy_1)^4,\quad g_2=x g_1x^{-1},\quad g_3=(y_1 x) g_1 (y_1 x)^{-1},\quad g_4 =(x y_1 x) g_1
(xy_1x)^{-1}.$$
Then,
$$g_1^2=1, \quad [g_i,g_j]=\pm 1\;\; \textrm{ for all } i,j=1,\ldots,4,\quad x g_1 x^{-1} = g_2,\quad$$
$$ x g_2 x^{-1}=g_1,\quad x g_3 x^{-1}=g_4,\quad x g_4 x^{-1}=g_3,\quad y_1 g_1 y_1^{-1} = g_2,\quad y_1 g_2 y_1^{-1} =g_3. $$
Moreover,  $y_1 g_3y_1^{-1} = -g_1g_4$, and $y_1 g_4 y^{-1} =-g_1 g_2$. Therefore, $\langle x,
y_1\rangle$ possesses a normal subgroup of order $2^5$. Hence,
$\langle \bar{x},\bar{y}_1 \rangle\neq \mathrm{PSU}_4(9)$.

\smallskip
{\bf Case 2.2.5:} If $(a_1,a_2,a_3,a_4)=\pm(1,1,1,-1)$, then take $g_1\ldots,g_5$
as in the case 2.2.4. Then, they satisfy the same conditions, with the only differences:
$y_1 g_3y_1^{-1}= g_1g_2g_3$, and $y_1 g_4 y^{-1} =g_3 g_4$. Therefore,
$\langle x, y_1\rangle$ possesses a normal subgroup of order $2^5$.
Hence, $\langle \bar{x},\bar{y}_1 \rangle\neq \mathrm{PSU}_4(9)$.

\smallskip
{\bf Remark.} In Cases 2.2.4 and 2.2.5,  by MAGMA calculations, 
$\langle x,y_1\rangle$ is soluble of order $2^8\cdot 3^2$.

\smallskip

{\bf Case 2.2.6:} In each case listed below we indicate explicitly two elements $u$ and
$g_1\in \langle x,y_1\rangle$ and set $g_2=u g_1 u^{-1}$, $g_3=u^2
g_1 u^{-2}$, $g_4=u^3 g_1 u^{-3}$, $g_5=u^4 g_1 u^{-4}$. A direct
computation shows that in each case the projective images of
$g_1,\dots,g_5$ satisfy the relations of Lemma \ref{pres}(vi).
\begin{center}
\begin{tabular}{ccccc}
$ (a_1,a_2,a_3,a_4)$ & $u$ & $g_1$ & $x$ & $y$ \\ \hline
$\pm(1,0,1,-1)$ & $(x y_1)^2$ & $(x y_1^2)^{-2}$ & $\mp t_1$ & $\xi^2 w_1$ \\
$\pm(1,-1,1,-1)$ & $(x y_1)^2$  & $(x y_1^2)^{-2}$ & $\pm t_1$ & $-w_1$ \\
$\pm(1,1,1,0)$ & $x y_1$ &  $(x y_1 x y_1^{-1})^2$ & $\mp t_2$ & $-w_2$ \\
$\pm(1,-1,1,0)$ & $x y_1$& $(x y_1 x y_1^{-1})^2$ & $\mp t_2$ & $\xi^2w_2$ \\
\end{tabular}
\end{center}
where
$$\left. \begin{array}{ll}
t_1= \xi^2 g_1^2 g_2 g_4 g_3 g_5 g_4 g_3,  & t_2=\xi^2 g_1^2 g_2 g_1 g_3 g_2 g_4 g_3,\\
w_1= g_1^2 g_2 g_1^2 g_4 g_3 g_2 g_1, & w_2=g_1^2 g_2 g_1 g_3 g_2 g_1 g_4 g_3 g_2 g_5.
     \end{array}\right. $$

As $\bar{x},\bar{y}_1 \in \langle \bar{g}_1,\dots,\bar{g}_5\rangle $, this proves that
$\langle \bar{x},\bar{y}_1 \rangle\cong\mathrm{Alt}(7)$.

\medskip

{\bf Case 2.2.7:} The remaining cases are $(a_1,a_2,a_3,a_4)=\pm(1,0,-1,0)$ or
$\pm(1,-1,0,0)$. First, notice that $H_1=\langle x, y_1 x
y_1^{-1}, y_1^2\rangle$ is a subgroup of $\langle x, y_1 \rangle$ of
index at most 2. (e.g., by induction on the length in $x$, $y_1$). Set $g=y_1xy_1^{-1}$.
Clearly, $h=x y_1^{-1} x y^2 x y_1 x\in H_1$. One can check that
$y_1^2 = -\xi^2 (ghx)^3hgx$, so  $\bar{y}_1^2 \in \langle \bar{x},
\bar{g}, \bar{h}\rangle$. Moreover, $\bar{x}$, $\bar{g}$,
$\bar{h}$ satisfy the following presentation for $\Alt(6)$ of Lemma \ref{pres}(v).

We note that in these cases MAGMA returns that $\langle x,y_1\rangle$ modulo the center is $\Alt(6).2$.

\smallskip
{\bf Case 2.3:} $Y=y_1$ with $a=\xi^5$, $b=0$.

Then $c_1=0$, $b_1\in \left\{1,-1\right\}$
and $b_2=\pm b_1$. If $b_2=-b_1$, substituting $v_2$ with $\xi v_2$ we may assume $b_2=b_1$ and, up to a scalar, $B=I$. The following conditions must hold:
$$r_1=-\alpha-i\alpha^3,\quad r_4=i\alpha^3-\alpha ,\quad \alpha\in \F_9; \quad r_3=-r_2^3.$$
We may assume $r_2\ne 0$, otherwise $\left\langle v_1,Xv_1\right\rangle$ is $\left\langle x, y_1\right\rangle$-invariant.
Consider
$$K=\left( \begin{array}{cccc}
0 & 0& 0& 1 \\
0 & 0& \pm 1 & 0\\
0 & \mp 1 & 0 & 0 \\
-1 & 0& 0& 0
\end{array}\right)$$
Then, $K=-K^T$ is non-singular and clearly $x^TKx=\mp K$.
If $r_4=\pm \xi^2 r_1$,  then
$y^T K y=K$. So, in this case, $\left\langle x,y_1\right\rangle$ is
contained in $\CSp_4(9)$.

Now, if $r_4\neq 0$, we take the involution
$$z=\left(\begin{array}{cccc} 0 & 1 & 0 & 0\\ 1 & 0 & 0 & 0 \\ 0 & 0 & 0 & 1 \\ 0 & 0 & 1& 0 \end{array}\right)$$
and consider $zx^\sigma z^{-1}=x$ and $zy^\sigma z^{-1}$ instead of $x$ and $y$.
In such a way we may assume $r_4=0$. Furthermore, conjugating our generators by
the diagonal matrix $\diag(1,r_2,1,r_2)$, we may consider only the following $4$ cases:
$r_1=\pm \xi^3$, $r_2=1$, $r_3=\pm 1$, $r_4=0$.

Now we analyze the group generated by $x$ and
$$
 y_1=\left(\begin{array}{cccc}
  \xi & 0      & r_1   & 1 \\
  0   & \xi^3  & r_3   & 0 \\
  0   & 0      & \xi^5 & 0 \\
  0   & 0      &  0    & \xi^7
\end{array}\right),
$$
where $r_1=\pm\xi^3$, $r_3=\pm 1$. The aim is to show that in all these cases $\langle x, y_1\rangle/Z\cong \Alt(7)$. In each case, proceeding as done in case 2.2.6, we take $g_1,\ldots,g_5$ such that their projective images satisfy the relations of Lemma \ref{pres}(iv).
\begin{center}
\begin{tabular}{cccccc}
$r_1$ & $r_3$ & $u$ & $g_1$ &  $x$ & $y$ \\\hline
$\pm\xi^3$ & $-1$ & $x y$ & $(xyxy^{-1})^2$ & $\mp t_2$  & $w_2$ \\
$\pm \xi^3$ &  $1$ & $(xy)^{-2}$  & $(x y^2)^2$  & $\pm t_1$  & $w_3$  \\
\end{tabular}
\end{center}
where $t_1$, $t_2$ and $w_2$ are as in case $2.2.6$ and
$w_3=\xi^2 g_1 g_2 g_1^2 g_3 g_2 g_1^2 g_5 g_4 g_3 g_2 g_1^2$.
Moreover, the projective images of $x$ and $y$ lie in
$\langle \bar{g}_1,\dots, \bar{g}_5 \rangle\cong \Alt(7)$.

\smallskip
{\bf Case 2.4:} $Y=y_1$ with $a=\xi^5$, $b=1$.

If $b=1$, then $B$ (hence $J$) is degenerate: in particular $b_2=1$,$b_1=c_1=0$ .

\medskip
{\bf Case 2.5:}  $Y=y_2$. Then, a first necessary condition so that $H\leq \SU_4(9)$ is
$c^3+c+1=0$, i.e. $c\in\left\{1,\xi^5,\xi^7\right\}$.

If $c=\xi^5$, then $B$ (hence $J$) is degenerate: in particular $b_1=1$,$b_2=c_1=0$.

If $c=\xi^7$, taking the involution $z$ of case 2.3, we obtain $zx^\sigma z^{-1} =x$ and
$$zy^\sigma z^{-1}=\left(\begin{array}{cccc}
\xi & 0 & r_4^3 & r_3^3\\
0 & \xi^3 & r_2^3 & r_1^3 \\
0 & 0 & \xi^7 & 1\\
0 & 0 & 0 & \xi^5
\end{array}\right).$$
So we may refer to the previous case  2.2.

If $c=1$, then in the Gram matrix $b_1=1$, $b_2=-1$ and $c_1=\xi^2$.
We have the following conditions:
$$\left\{ \begin{array}{l}
(r_1-r_3)^3+\xi^2(r_1-r_3) = 0\\
r_1+r_3-(r_2-r_4)^3=0\\
(r_2+r_4)+\xi^2(r_2+r_4)^3= 0
          \end{array}
\right.$$
Moreover, if $r_1-r_3+\xi^2(r_2+r_4)=0$ or $
r_1-r_4+\xi(r_2+\xi^2r_3)=0$, then $H\leq \CSp_4(9)$. We obtain $36$
possibilities for $(r_1,r_2,r_3,r_4)$.

{\bf Case 2.5.1:} If $(r_1,r_2,r_3,r_4)$ is one of the following
$$\left.\begin{array}{llll}
\pm (0 , \xi , 0 , \xi ) & \pm (0, \xi, \xi^7, \xi^5) &
\pm (\xi, \xi^3, 1,-1) & \pm( 1, 1, \xi, \xi^7) \\
\pm (1,1,\xi^6,\xi^6) & \pm ( \xi^2, \xi^6, -1,1) &
\pm (\xi^3, 0, \xi^7, 0) & \pm (\xi^3, \xi^5, \xi^3, 0)
\end{array}\right.$$

then, denoting by   $\bar{x}$,   $\bar{y}_2$ the projective images of
$x,y_2$ satisfy the presentation of Lemma \ref{pres}(i).
Consequently, $\langle \bar{x},\bar{y}_2\rangle \cong
\mathrm{PSL}_3(4)$.

\medskip
{\bf Case 2.5.2:} If $(r_1,r_2,r_3,r_4)$ is one of the following
$$\pm (1,\xi^3,1,\xi^6)\qquad \pm (\xi^2,-1,\xi^5,1),$$
then, taking $g_1=(xy_2)^4$, $g_2,\ldots,g_5$ as done in case 2.2.5, we obtain that $\langle x,y_2\rangle$ possesses a normal subgroup of order $2^5$.

\medskip
{\bf Case 2.5.3:} In the following cases we take $g_1\ldots,g_5$ as done in case 2.2.6:
\begin{center}
\begin{tabular}{ccccc}
$(r_1,r_2,r_3,r_4)$ &  $u$ & $g_1$ & $x$ & $y$ \\ \hline
$\pm (0,0,\xi^7,\xi)$ & $xy$ & $(xyxy^{-1})^2 $ & $\pm t_2$ & $\xi^2w_2$ \\
$\pm (\xi^7,\xi^5,0,0)$ & $xy$ & $(xyxy^{-1})^2 $ & $\pm t_2$ &  $-\xi^2w_2$  \\
$\pm (1,\xi^3,\xi^6,\xi^7) $ & $(xy)^2$ & $(xy^2)^{-2}$ & $\pm t_1$ & $w_1$\\
$\pm (\xi,\xi^2,\xi,-1)$  &  $(xy)^2$ & $(xy^2)^{-2}$ & $\pm t_1$& $w_1$ \\
\end{tabular}
\end{center}
where $t_1,$ $t_2$, $w_1,$ $w_2$ are as in case $2.2.6$.
The projective image of $\langle x,y_2\rangle$ is thus isomorphic to $\Alt(7)$.

\medskip

{\bf Case 2.5.4:}
Finally, if $(r_1,r_2,r_3,r_4)$ is one of the following
$$\pm (\xi,1,\xi^6,-1) \qquad \pm(1,\xi^2,\xi,\xi^6)\qquad \pm(1,\xi^2,1,\xi^7)\qquad \pm (\xi^2,\xi^7,\xi^2,1), $$
we proceed as done in case 2.2.7, taking $g=-y^3 x y^{-1} (xy)^2$ and $h=-\xi^2y^2$.
Thus, the projective image of $\langle x,y_2\rangle$  isomorphic to $\Alt(6).2$.
\end{proof}

\section{Further assumptions}
\label{sec:further}

From now on we suppose $p>0$, $x,y$ defined as in \eqref{generators}, with:
\begin{equation}\label{assumption}
r_1=r_3=0,\quad r_4\neq 0.
\end{equation}
Under these assumptions, formulas \eqref{general car xy} and \eqref{general car xy^-1} become respectively:
\begin{equation}\label{car xy}
\begin{array}{cccccc}
\chi_{xy}(t)\hfill &=& t^4-dr_4t^3-dst^2-r_2t+1\hfill ; \\
\noalign{\smallskip}
\chi_{(xy)^{-1}}(t)\hfill &=& t^4-r_2t^3-dst^2-dr_4t+1\hfill .
\end{array}
\end{equation}
Moreover we set:
\begin{equation}\label{def: H}
H=\left\langle x,y\right\rangle .
\end{equation}
Lemma \ref{reducible} gives rise to the following:
\begin{corollary}\label{newreducible} $H$ is a reducible over $\F$ if and
only if one the following conditions holds:

${\rm (i)}$ \enskip  $r_2=-\epsilon^{\pm 1} r_4$;

${\rm (ii)}$\enskip $r_2 + r_4=\pm (2- s)\sqrt d$.
\end{corollary}

\begin{remark}\label{conjugate} Let $H$ be absolutely irreducible.
By Remark~\ref{rigid triple}, $xy$ has a unique similarity invariant.
It follows that, for any field automorphism $\sigma$, the matrices
$(xy)^\sigma$ and $(xy)^{-1}$ are conjugate if and only if they have
the same characteristic polynomial.
\end{remark}
When dealing with the unitary groups we denote by $\sigma$ 
the Frobenius map $\alpha\mapsto  \alpha^{q}$
of $\F_{q^2}$. 

\begin{theorem}\label{classic}

{\rm (i)} If $H\le \mathrm{CO}_4(q)$ or $H\le \mathrm{CSp}_4(q)$,
then $r_2=\pm dr_4$;

{\rm (ii)} If $H\le \mathrm{CU}_4(q^2)$, then $r_2=\pm d r_4^\sigma$
and $s^\sigma=s$.

{\rm (iii)} In particular,  if $r_2=0$ and $r_4\neq 0$, then $H$ is
not contained in any of the groups $\CSp_4(\F)$, $\CO_4(\F)$,
$\CU_4(\F)$.
\end{theorem}

\begin{proof}
Let
 \begin{equation}
 \label{eq:xy-forms}
 x^T J x^{\sigma_1}=\lambda J, \qquad
 y^T J y^{\sigma_1}=\mu J,
 \end{equation}
where $J$ is a nondegenerate  symmetric or skew-symmetric or hermitian
form and $\sigma_1=\mathrm{id}$ or, respectively, $\sigma_1=\sigma$.
Write $$ J=\left(
 \begin{array}{cc}
  j_1 & j_2 \\
  j_3 & j_4
 \end{array}
  \right),
$$
where $j_1$, $j_2$, $j_3$, $j_4$ are $2\times2$ matrices.
Clearly, $\lambda=\pm 1$ and $j_3=d\lambda j_2$, $j_4=\lambda j_1$.
If $\mu\neq 1$, then it follows from the second relation in \eqref{eq:xy-forms} that
$j_1=0$ and
$$
  j_2 \left(
 \begin{array}{cc}
  0 & -1 \\
  1 & s^{\sigma_1}
  \end{array}
  \right) = \mu j_2.
$$
In particular, the rank of $j_2$ is at most 1 and $J$ is degenerate.
Therefore, $\mu=1$. Consequently, $(xy)^T J (xy)^{\sigma_1}=\pm J$ and
$(xy)^{-1}$ is conjugate to $\pm(xy)^{\sigma_1}$. Using \eqref{car xy} we
prove (i) and (ii). The final claim is now obvious.
\end{proof}

If $H$ is absolutely irreducible, the previous Theorem can be
(partially) reverted. Namely, we have the following result.

\begin{theorem}
 \label{classic_positive}
Assume that $H$ is absolutely irreducible. 

{\rm(i)} If
$\F_{q^2}=\F_p\!\left[r_4\right]$, $s\in \F_q$, and $r_2= d r_4^\sigma$, then
$H\le \SU_4(q^2)$.

{\rm(ii)} If $\F_{q}=\F_p\!\left[s, r_4\right]$, $r_2=r_4$ and $d=1$,
then $H\le \SO^{\pm}_4(q)$.

{\rm(iii)} If $p\neq 2$, $\F_{q}=\F_p\!\left[s, r_4\right]$, $r_2=r_4$ and $d=-1$,
then $H\le \mathrm{CSO}^{\pm}_4(q)$, $H\not\le \SO^{\pm}_4(q)$.

{\rm(iv)} If $\F_{q}=\F_p\!\left[s, r_4\right]$, $r_2=-r_4$ and $d=1$, then  $H\le
\CSp_4(q)$ and, if $p\neq 2$, $H\not\le \Sp_4(q)$.

{\rm(v)} If $\F_{q}=\F_p\!\left[s, r_4\right]$, $r_2=-r_4$ and $d=-1$, then  $H\le
\Sp_4(q)$.
\end{theorem}

\begin{proof}
(i) By the  assumptions and by \eqref{car xy} together with
Remark~\ref{conjugate} we have that $(xy)^\sigma$ is conjugate to $(xy)^{-1}$.
Our claim follows from Theorem~\ref{forms}(ii).

(ii) See Lemma \ref{negative}.

(iii)--(v) Unfortunately, using only Theorem~\ref{forms}(i) or
Corollary~\ref{conformal}, we cannot distinguish symmetric and
skew-symmetric forms. For the symmetric case, necessary conditions
can be deduced from Lemma~\ref{negative}. However, we are still
unable to deal with the symplectic groups in this way. Thus we prefer
to present an alternative approach giving explicit forms. Namely, let
$r_2=\lambda r_4$, $\lambda=\pm1$ and set
$$
J=\left(
 \begin{array}{cccc}
   2 i_1 & i_2             &  r_4 & r_4 \\
  \lambda i_2    & 2 i_1 &  r_4 & r_4 \\
   \lambda r_4    & \lambda r_4     & 2d i_1 & d\lambda i_2 \\
  \lambda r_4     & \lambda r_4     & d i_2            & 2d i_1
  \end{array}
  \right)
$$
where $i_2=2-s-2 i_1$ and $i_1=\frac{2-s-dr_4^2}{s+2}$ if
$\lambda=1$, $i_1=0$ if $\lambda=-1$. Notice that, for $\lambda=1$,
by Corollary~\ref{newreducible}(i) we have $\epsilon\neq-1$, i.e.,
$s\neq -2$. Thus, $i_1$ is well defined.

By a straightforward calculation, $x^T J x = \lambda d J$ and $y^T J
y =J$. Moreover, for $\lambda=1$,  $J$ is symmetric and
   $$\det J=\frac{((s - 2)^2 - 4 dr_4^2)^3}{(s+2)^2}\neq 0,$$
by irreducibility and Corollary~\ref{newreducible}(ii). For
$\lambda=-1$,  $J$ is skew-symmetric and $\det J=(s - 2)^4\neq 0$
again by irreducibility and Corollary~\ref{newreducible}(ii).
\end{proof}

\begin{lemma}\label{powers} Assume $(xy)^h=\rho I$, for some $\rho \in \F$, with $h>0$. Then:

{\rm (i)} $h> 4$.

{\rm (ii)} $h=5$  only when  $s=\pm 1$,  $r_4^2=-sd$, $r_2=-sr_4$.
Moreover, if $s=-1$, the projective image of $H$ is isomorphic to $\Alt(5)$.
If $s=1$, then $H$ is either reducible over $\F$ or contained in the conformal symplectic group. 

{\rm (iii)} $h\ne 6$, unless $\left\langle x,y\right\rangle$ is reducible.

{\rm (iv)} if $h=8$, the following relations hold:
$$\left.
\begin{array}{c}
r_4^2+\rho r_2^2=-ds(1+\rho)\hfill\\
sdr_4^2 + d (1-\rho)r_2r_4 =-s^2-\rho+1\hfill\\
-\rho r_2^4 - 3ds\rho r_2^2- 2d\rho r_2r_4 =\rho s^2 -\rho+1\hfill
\end{array}
\right.
$$
\end{lemma}

\begin{proof}

{\rm (i)} $[xy]_{1,4}=-1$, $[(xy)^2]_{1,2}=-d$, $[(xy)^3]_{1,3}=-d$, $[(xy)^4]_{1,3}=-r_4$.

\smallskip
{\rm (ii)}  Let $D=(xy)^5$. Then
$$D_{1,3}=-dr_4^2-s=0,\quad D_{2,1}=r_4 ds+r_2 d=0,\quad D_{3,1}=r_2r_4-d=0.$$
It follows that $s=-d r_4^2$, $r_2=-s r_4$, $r_4^4=1$. In particular $s=\pm 1$,
whence the conditions in the statement. On the other hand, if these conditions hold,  then $(xy)^5=-d r_4I$.
In particular, if $s=-1$, then $r_4^2=d$, $r_2=r_4$ and $H/(H\cap Z)\cong \Alt(5)$ 
by  Lemma \ref{pres}(i). If $s=1$, then $r_4^2=-d$, $r_2=-r_4$ and either $H$ is reducible over $\F$ or $H\le \CSp_4(q)$ by
Theorem \ref{classic_positive}(iv),(v).

\smallskip
{\rm (iii)} 
Let $D=(xy)^6$.
Then, $D_{4,1}=d r_4^3+ 2s r_4 + r_2=0$, i.e. $r_2=-d r_4^3-2s r_4$. Under this hypothesis, we have
$$D_{1,2}=d r_4^4 + s r_4^2- d s^2+d=0,\quad D_{3,1}= -r_4^5-3 d s r_4^3 - 2 r_4 s^2- r_4=0.$$
It follows that $r_4^4+sdr_4^2-s^2+1=0$ and $r_4^4+3dsr_4^2+2 s^2 +1=0$, whence
$2ds r_4^2+3 s^2=0$. Thus, either $s=0$ or $2dr_4^2+3s=0$.
In the first case, $y$ has order $4$, $r_2=-d r_4^3$ and $r_4^4=-1$.
It follows that $r_4^2=\pm \epsilon$ and so $H$ is reducible by  Corollary \ref{newreducible}(i).
In the second case we have $p\neq 3$ by the assumption $r_4\ne 0$, hence
$s=-\frac{2}{3}dr_4^2$, $r_2=\frac{1}{3}d r_4^3$ and $r_4^4=9$.
It follows either $s=-2d$, $r_2=d r_4$ or $s=2d$, $r_2=-d r_4$.
This implies that $\epsilon=\pm 1$ and, in any case, that  $r_4=-\epsilon r_2$. Again
$H$ is reducible by Corollary \ref{newreducible}(i).
(Note that in both cases $(xy)^6$ is scalar).

\smallskip
{\rm (iv)} Let $D=(xy)^4-\rho(xy)^{-4}$. Then the system of equations $D_{1,2}=D_{2,2}=D_{1,1}=0$ is equivalent to
the system in the statement.

\end{proof}

\section{Conditions under which $H\leq M\in {\cal C}_2$}\label{C2}
\begin{lemma}\label{wreath} If $H$ is absolutely irreducible, it
does not stabilize any $2$-decomposition of $\F^4$.
\end{lemma}
\begin{proof}
Assume, by contradiction,  that $H$ stabilizes a $2$-decomposition
$\F^4=V_1\oplus V_2$.
We claim that $y$ fixes $V_1$ (and $V_2$). To this purpose set
$W:=\left\{w\in \F^4\mid y^2w=w\right\}$.
As $y^2$ fixes $V_1$ and $V_2$, we have
$$W=(W\cap V_1)\oplus (W\cap V_2).$$
From $\left\langle e_1,e_2\right\rangle\ \leq W$ and $y^2\neq I$,
it follows $\dim W =2,3$.
If $\dim W =2$, then $W=\left\langle e_1,e_2\right\rangle$.
If $\dim W =3$, we may suppose $V_1\leq W$. So, in both cases,
there exists a non-zero vector $v_1\in \left\langle e_1, e_2\right\rangle \cap V_1$.
We conclude that $V_1$ is fixed by $y$. By the irreducibility of $H$
we must have $V_2=(xy)V_1$ and $V_1=(xy)V_2$, a contradiction as
$\tr{xy} =dr_4\neq 0$.
\end{proof}

\begin{lemma}\label{monomial} If $H$ is absolutely irreducible, it
does not stabilize any $1$-decomposition of $\F^4$.
\end{lemma}
\begin{proof}
By contradiction, let $H$ stabilize a $1$-decomposition
$\F^4=V_1\oplus V_2\oplus V_3\oplus V_4$. 
Then, for some $g\in \GL_4(\F)$, the group $H^g$ is contained
in the standard monomial group $D\Sym(4)$,
where $D$ consists of the diagonal matrices. From $\tr{xy} =dr_4\neq 0$,
it follows that the image of $(xy)^g$ in $\Sym(4)$ is neither a $4$-cycle, nor the product
of two $2$-cycles.
Thus $(xy)^j\in D^{g^{-1}}$ for some $j=1,2,3$.
For the same $j$, also $(yx)^j= \left((xy)^j\right)^x$ must be in $D^{g^{-1}}$.
In particular $A:=(xy)^j(yx)^j-(yx)^j(xy)^j=0$. We claim that $j\neq 1,2$.
Indeed, if $j=1$, then $A_{23}=-r_4d\neq 0$ and, if $j=2$, then
$A_{21}=-r_4^2d\neq 0$. We are left to consider the case $(xy)^3\in D^{g^{-1}}$
and $xy\not\in D^{g^{-1}}$.
The trace $dr_4$ of $xy$ must be an eigenvalue of $xy$.
Moreover $\chi_{xy}(t)$ factorizes as $(t-dr_4)(t^3-dr_4^{-1})$.
It follows from \eqref{car xy}  that $r_2r_4=d$ and $s=0$.
From $A_{11}=- r_4^2+ d$ we get $r_4=\pm \sqrt d$.  Thus $H$ is reducible
by Corollary \ref{newreducible}(ii).
\end{proof}

\section{Conditions under which $H\leq M\in {\cal C}_6$}\label{C6}

We refer to \cite[Section 4.6, pages 148--155]{KL} for details.
A maximal subgroup $M\in {\cal C}_6$ is the normalizer 
of an absolutely irreducible symplectic-type $2$-group $N$. It follows that $p\ne 2$
and the centre $Z$ of $N$ is scalar. 
Moreover $N$ has exponent $4$, and the factor group $N/Z$ is elementary abelian
of order $2^4$.
Note that, in  every non identity coset of $Z$ in $N$, the
elements of the same order are opposite to each other.
Since conjugation preserves the order of elements, it follows that,
for all $g\in C_M\left(N/Z\right)$ and all $n\in N$:
\begin{equation}\label{derived}
n^g=\pm n.
\end{equation}
In particular $N'=\left\langle -I\right\rangle$, since $N\leq C_M\left(N/Z\right)$,
hence $[n_1,n_2]=n_1^{-1}n_1^{n_2}=n_1^{-1}\left(\pm n_1\right)=\pm I$.

Now, the conjugation action of $\GL_4(\F)$ on $\Mat_4(\F)$ induces a homomorphism:
\begin{equation}\label{mu}
\mu: \GL_4(\F) \to \GL_{16}(\F).
\end{equation}
As $N$ is absolutely irreducible, its linear span $\F N$ coincides with
$\Mat_4(\F)$.
Hence any transversal $T$ of $Z$ in $N$ is a basis for $\Mat_4(\F)$.
Noting that, when $Z$ has order $4$, we may choose $T$ consisting of involutions,
we  can assume that $\mu(M)$ consists of monomial matrices with entries $0,\pm 1$.

Now $\mu$ induces a homomorphism
\begin{equation}\label{tau}
\tau: M\to {\rm Aut}\left(N/Z\right)
\end{equation}\label{symplectic}
\noindent whose kernel is $C_M\left(N/Z\right)$.
If we identify  $N/Z$ with $\F_2^4$, and set
$(Zn_1, Zn_2)=0$ if $[n_1,n_2]=I$,  $(Zn_1, Zn_2)=1$ if $[n_1,n_2]=-I$,
we define a non-degenerate symplectic form: indeed $[n,N]=I$
only if $n\in Z$. As this form is preserved by $\tau (M)$, we have
$\tau(M)\leq \Sp_4(2)$. Note that $\mu(N)$ consists of diagonal matrices.
Let $n_1\in T$ be such that $Zn_1\ne Z$. Then space orthogonal to $Zn_1$ has dimension $3$.
So there are $2^3=8$ elements $n\in T$ for which $n^{n_1}=n$. It follows that the Jordan form
of $\mu(n_1)$ is diag$(1^8,(-1)^8)$.

Finally, let $g\in$ Ker $\tau = C_M\left(N/Z\right)$.
It follows from \eqref{derived} that $g^2$ centralizes $N$, hence $g^2\in Z$
by the absolute irreducibility of $N$. Thus (Ker $\tau)/Z$
is an elementary abelian $2$-group.

From (Ker $\tau)/Z$ normal in $M/Z$, we have that (Ker $\tau)/N$ is a normal 2-subgroup of $M/N$.
But for the groups that we are considering, $M/N$ is isomorphic
to one of the groups $\Alt(5)$,  $\Alt(6)$, $\Sym(5)$,
$\Sym(6)$. Since in all these groups the only normal 2-subgroup is the identity,
we conclude that (Ker $\tau)=N$.

Table \ref{tab:1} is deduced from  the natural action of $\Sp_4(2)$ on $\F_2^4$.
In the last column we consider that case in which $g\in M$ is such that $\tau(g)$
belongs to the corresponding class. Note that $g$ is not unique, nevertheless
this column is consistent by the previous considerations.

Note that, for each $g\in M$, at least one diagonal entry of $\mu(g)$
is $1$, since $\lambda I\in T$, for some $\lambda$, 
and $(\lambda I)^g=\lambda I$.

\begin{table}[th]
\begin{tabular}{ccc}\\
\noalign{\medskip}
{Conj. classes of}\ $\Sp_4(2)$&{Orbit structure on} $\F_2^4$&$\tr{\mu(g)}$\\
\hline
\noalign{\medskip}
$2_1$, $2_2$\hfill &$1^4$, $2^6$\hfill& $4$, $\pm 2$\\
\noalign{\medskip}
$2_3$\hfill &$1^8$, $2^4$\hfill&$8$, $\pm 6$, $\pm 4$, $\pm 2$\\
\noalign{\medskip}
$3_1$\hfill &$1^4$, $3^4$\hfill&$4$, $\pm 2$\\
\noalign{\medskip}
$3_2$\hfill &$1$, $3^5$\hfill& $1$\\
\noalign{\medskip}
$4_1$, $4_2$\hfill &$1^2$, $2$, $4^3$ \hfill& $2$\\
\noalign{\medskip}
$5$\hfill &$1$, $5^3$\hfill&$1$  \\
\noalign{\medskip}
$6_1$\hfill &$1^2$, $2$, $3^2$, $6$\hfill & $2$\\
\noalign{\medskip}
$6_2$\hfill & $1$, $3$, $6^2$\hfill &$1$
\end{tabular}

\caption{{}\quad }   \label{tab:1}
\end{table}

\begin{table}[th]
\begin{tabular}{ccc}
\noalign{\medskip}
$g$\hfill &&$\tr{\mu(g)}$\\
\hline
\noalign{\medskip}
$y$\hfill &&$(s+2)^2$\\
\noalign{\medskip}
$xy$\hfill &&$dr_2r_4$\\
\noalign{\medskip}
$(xy)^2$\hfill &&$r_2^2r_4^2 + 2ds(r_2^2 + r_4^2) + 4s^2$ \\
\end{tabular}
\caption{{}\quad }   \label{tab:2}
\end{table}

\begin{lemma}\label{oddorder} Let $g\in M$. If no power of $\mu(g)$ has the eigenvalue $-1$,
in particular if $g$ has odd order, then $\dim C_{\Mat_4(\F)}(g)$ is equal to the
number of orbits of $\tau(g)$ on $\F_2^4$.
\end{lemma}
\begin{proof}
Given $Zn\in N/Z$, let $\left(Zn, Zg^{-1}ng, \dots , Zg^{-h+1}ng^{h-1}\right)$ be its orbit
under $\tau(g)$. If no power of $\mu(g)$ has the eigenvalue $-1$, by \eqref{derived}
we have $g^{-h}ng^h=n$. Thus $\left(n, g^{-1}ng, \dots , g^{-h+1}ng^{h-1}\right)$ is an orbit
of $\mu(g)$ consisting of linear independent matrices over $\F$.
In other words we can take a transversal of $Z$ in $N$ which is the union of orbits of
$\tau(g)$. It follows that a matrix $z=\sum \lambda_i n_i$, with $n_i\in T$,  is centralized by $g$ if and only if
elements in the same orbit have the same coefficients. This means that the orbits sums are a basis
for $C_{\Mat_4(\F)}(g)$, and  the claim follows.
\end{proof}

\begin{lemma}\label{evenorder}  If $H\leq M$ and $y$ has even order $k=2m$, then
$y^m\not\in N$. In particular $Ny$ has order $k$.
\end{lemma}
\begin{proof}
Write
$y=\left(
\begin{array}{cc}
I&R\\
0&S
\end{array}\right),\ R=\left(
\begin{array}{cc}
0&r_2\\
0&r_4
\end{array}\right),\ S=\left(
\begin{array}{cc}
0&-1\\
1&s
\end{array}\right)$. Note that $S$ has order $k$ and
$S^m$ has eigenvalues $\epsilon^{\pm m}=1$ only if $k=2p$, with $p$ odd.
 But 
$$y^m=\left(
\begin{array}{cc}
I&R\left(I+S+\dots +S^{m-1}\right)\\
0&-I
\end{array}\right)$$
where $\Sigma=I+S+\dots +S^{m-1}$ is non singular as
it is a factor of $S^m-I$ and $S^m$ does not have the eigenvalue $1$.
Thus $Y=R\Sigma$ has rank 1.
Assume, by contradiction,
that $y^m\in N$. Then  $(y^m)^x\in N$. Moreover $y^m(y^m)^x=\lambda (y^m)^xy^m$, i.e.,
\begin{equation}\label{ymymx}
\left(
\begin{array}{cc}
-I+dY^2&Y\\
-dY&-I
\end{array}\right)\ =\ \lambda \left(
\begin{array}{cc}
-I&-Y\\
dY&dY^2-I
\end{array}\right).
\end{equation}
We conclude $\lambda=-1$ and $Y^2=2dI$.
 But this is a contradiction as $Y^2$ has rank $\leq 1$ and $p\ne 2$.
The last claim follows from the fact that $N$ is a 2-group.
\end{proof}

\begin{lemma}\label{kappa} If $H\leq M$, then $k=3$, i.e., $s=-1$.
\end{lemma}
\begin{proof}

Lemma \ref{evenorder} and consideration of the conjugacy classes of $\Sp_4(2)\cong \Sym(6)$ give
$k\in \{3,4,5,6\}$.

Let $k=4$, i.e., $s=0$. By the third column of Table \ref{tab:1}, $\tr{\mu(y)}=2$
and, by Table \ref{tab:2}, $\mu(y)$ has trace $4$: a contradiction as $4\not\equiv 2 \pmod p$.

Let $k=5$. By the second column of Table \ref{tab:1}, $\tau(y)$ has $4$ orbits.
Hence, by Lemma \ref{oddorder}, $C_{\Mat_4(\F)}(y)$ should have dimension $4$, in contrast
with \eqref{centralizers}.

Finally, let $k=6$ i.e., $s=1$. Then $\mu(y)$ has trace $9\not\equiv 1\pmod p$.
It follows that $\tau(y)$ cannot be of type $6_2$.
On the other hand, $\tau(y)$ cannot be of type $6_1$. Indeed, in this case,
$\tau(y^2)$ would be of type $3_1$, which gives the contradiction $\dim C_{\Mat_4(\F)}(y^2)=8$.
\end{proof}

\begin{lemma} \label{cond C6}
Assume $H\leq M$. Then $s=-1$ and one of the following holds:

$\bullet$  $\tau(xy)$ has order $5$, $r_2=r_4=\pm \sqrt d$ and the projective image of $H$ is isomorphic to $\Alt(5)$;

$\bullet$ $\tau(xy)$ has order $6$, $r_2=2/(dr_4)$,  $r_4=\pm \sqrt {2d}$ and $H$ has order $2^63^2$.
\end{lemma}

\begin{proof}

Let $m$ be the order of $xy$ mod $N$, i.e., the order of $\tau(xy)$, as $N=$ Ker $\tau$.
It follows that $\mu\left((xy)^{m}\right)$ is either scalar or has Jordan form diag$\left(1^8,(-1)^8\right)$.
Note that, by Lemma \ref{powers}, we have $m\neq 2,3$. Moreover, when $(xy)^m$ is not scalar,
also $\mu\left((xy)^{m}\right)$ is not scalar.

Set $T_1=\tr{\mu(xy)}$, $T_2=\tr{\mu((xy)^2)}$. By Table 3, for $s=-1$:
\begin{equation}\label{Traces}
dr_2r_4=T_1,\quad (r_2+r_4)^2= d(T_1^2+4T_1-T_2+4)/2
\end{equation}

$\bullet$ Assume  $m=4$. Then, by Table 2,  $T_1\in \{2,0\}$.
From $N^2=N'=\left\langle -I\right\rangle$ we have $(xy)^8=\pm I$.
Recall that $s=-1$, by Lemma \ref{kappa}.
So the system of equations of Lemma \ref{powers}(iv)  must be satisfied
with $\rho =\pm 1$,  $s=-1$.
If $T_1=0$, then $r_2=0$, and the third equation gives the contradiction $0=1$.
If $dr_2r_4=T_1=2$, the system has the unique solution
$r_4^2=r_2^2=d(2-\rho)$, $r_2^4=2(1-2\rho)$. From $(2-\rho)^2=2(1-2\rho)$ we obtain $p=3$.
Thus $\rho=1$, $r_4=\pm \sqrt d$. It follows $r_2=-r_4$ and
$H$ is reducible by Corollary \ref{newreducible}(i).

\smallskip
$\bullet$ Assume  $m=5$. Then $T_1=T_2=1$. From \eqref{Traces} we get
$r_2=d/r_4$, $(r_2+r_4)^2=4d$. It follows $r_2=r_4=\pm \sqrt d$. By Lemma \ref{pres}(i),
the projective image of $H$ is isomorphic to $\Alt(5)$.

\smallskip
$\bullet$ Assume  $m=6$.
Note that $(xy)^6$ cannot be scalar by Lemma \ref{powers}.
If $\tau(xy)\in 6_1$, then $\mu(xy)^6=$ diag$(1^2,\alpha^2 ,1^6, \beta^6)$ and
if $\tau(xy)\in 6_2$, then $\mu(xy)^6=$ diag$(1,1^3, \alpha^6 , \beta^6)$.
As $(xy)^6\in N\setminus Z$,  by the above discussion we must have $\tau(xy)\in 6_1$, and
$\alpha=\beta=-1$. The condition $\alpha=-1$ gives $T_2=0$.

\smallskip
{\bf Case } $T_1=2$, whence $r_2=2/(dr_4)$.
As $T_2=0$, from \eqref{Traces} we get $(r_2+r_4)^2=8d$. It follows
$r_4=\pm \sqrt {2d}$.  Set $a=[x,y]$, $b=[x,y^2]$, $Q_8=\left\langle a^3,b^3\right\rangle$.
Then $Q_8$ is a normal subgroup of $H$, isomorphich to the quaternion group of order $8$.
From $(ab)^2\in Q_8$ we get that $\left\langle a,b\right\rangle/Q_8\cong \Alt(4)$.
As $x$ and $y$ commute mod $\left\langle a,b\right\rangle$, we conclude that $H$ has order $2^63^2$.

\smallskip
{\bf Case } $T_1=0$. By \eqref{Traces} we have $r_2=0$, $r_4=\pm \sqrt {2d}$.
After substitution of these values, the entry $(1,3)$ of $(xy)^{12}$ becomes
$-8r_4$, a contradiction.
\end{proof}

\section{Conditions under which $H\leq M \in \mathcal{S}$}\label{S}

A maximal subgroup $M$ in the class $\mathcal{S}$ is such that $M/Z$ has a 
unique minimal normal subgroup, which is an absolutely irreducible non-abelian simple group
(cf. \cite[p. 171]{GPPS}). Hence $Z=Z(M)$ is scalar.
Table \ref{tab:3} below describes the possibilities which arise for the groups 
in which we are interested (see \cite {K}). 

\medskip
\begin{table}[!ht]

\begin{tabular}{c|c|c|c}
$M/Z$ & $G$&Conditions\ under which\\
&&  $M/Z$\ is maximal in $G$&$\#$ Conj. Classes\\
\hline\\
$\Alt(7)$\hfill&$\PSL_4(q)$ \hfill& $q=p\equiv 1,2,4 \pmod{7}$\hfill&$(4,q-1)$\\
\hfill  &$\PSU_4(q^2)$\hfill & $q=p\equiv 3,5,6 \pmod{7}$ &$(4,q+1)$\hfill \\
\hfill & $\PSp_4(q)$\hfill & $q=7$ & 1 \hfill \\
\noalign{\medskip}\hline
$\PSp_4(3)$\hfill   &$\PSL_4(q)$\hfill & $q=p\equiv 1 \pmod{6}$\hfill&$(4,q-1)$\\
{}   &$\PSU_4(q^2)$\hfill & $q=p\equiv 5 \pmod{6}$\hfill&$(4,q+1)$\\
\noalign{\medskip}
\hline
 $\PSL_3(4)$\hfill   &$\PSU_4(q^2)$\hfill & $q=3$\hfill&$2$\\
\noalign{\medskip}\hline
$\PSL_2(q)$\hfill   &$\PSp_4(q)$\hfill & $p\geq 5$,\ $q\geq 7$\hfill&$1$\\
\noalign{\medskip}\hline
$\Alt(6)$\hfill   &$\PSp_4(q)$\hfill & $q=p\equiv 2,\pm 5 \pmod {12}$\hfill&1\\
$\Sym(6)$\hfill   &$\PSp_4(q)$\hfill & $q=p\equiv \pm 1 \pmod {12}$\hfill&2\\
\hline
\end{tabular}
\caption{{}\quad }  \label{tab:3}
\end{table}

In view of Lemma \ref{powers}(ii), in the following two Lemmas it is convenient to suppose 
that $(xy)^5$ is non-scalar.

\begin{lemma}\label{A7}
Assume that $H\leq M$, with $M/Z\cong  \Alt(7)$. If $H$ is absolutely irreducible and 
$(xy)^5$ is non-scalar, then
$s=-1$,
\begin{equation}\label{statement}
\begin{array}{c}
r_4=di^h(\omega^4 + \omega^2 + \omega + 1)=
d i^{h}\left(\pm \sqrt{-7}+1\right)/2,\hfill\\
\noalign{\smallskip}
r_2= -i^{3h}(\omega^4 + \omega^2 + \omega)=-i^{3h}\left(\pm \sqrt{-7}-1\right)/2
\end{array}
\end{equation}
where $i^2=-1$ and $\omega$ is a suitable primitive $7$-th root of unity
if $p\neq 7$, $\omega=1$ if $p=7$. 

Moreover $p\ne 2$ and $h=0,2$ if $p=7$ or $p\equiv 11,15,23 \pmod {28}$,
$h=0,1,2,3$ otherwise. 

In particular the projective image of $H$ is $\PSL_2(7)$.
\end{lemma}

\begin{proof}
By Lemma \ref{powers} and our assumption, the projective image of
$xy$ can only have order $7$. Every element of order $7$ in $M/Z$
is conjugate to its square and its fourth power. It follows that for $p\neq 7$,
the Jordan form of $xy$ must be $i^h\cdot \diag(1,\omega, \omega^2,\omega^4)$
where  $\omega$ is a suitable primitive  $7$-th root of $1$ and $h=0,1,2,3$. 
If $p=7$, a scalar multiple
of $xy$ is unipotent.
Thus, for any $p$, the characteristic polynomial of $xy$ is
$$t^4 - i^h(\omega^4 + \omega^2 + \omega + 1)t^3 - i^{2h}t^2 + i^{3h}(\omega^4 + \omega^2 + \omega)t + 1.$$
Comparison with \eqref{car xy}, gives the relations \eqref{statement}
and the further condition $ds=i^{2h}$.
In particular $ds=\pm 1$, hence $s=\pm 1$. If $s=1$, then $r_2+r_4=i^{3h}$, which
is excluded by Corollary \ref{newreducible}(ii). Thus $s=-1$ and, in particular, $p\neq 2$.

If $h$ is odd, then $i^h=dr_4-r_2\in \F_p\!\left[r_2,r_4\right]$.
According to Table 3 this may happen in the unitary case or in the case $\PSL_4(p)$,
$p\equiv 1\pmod 4$.

Finally for all cases listed in \eqref{statement}
the projective images $\bar x,\bar y$ of $x$ and $y$, respectively,
satisfy the presentation of $\PSL_2(7)$ given in Lemma \ref{pres}(ii).
\end{proof}

\begin{lemma}\label{PSp(4,3)}
Assume that $H\leq M$ with $M/Z\cong  \PSp_4(3)$.  Then $p\ne 2,3$.
Moreover, if $H$ is absolutely irreducible and $(xy)^5$ is non-scalar, then $s=0$ and, for some $h=0,1,2,3$:
\begin{equation}
r_2  = i^{-h}\omega, \quad r_4  =  d i^h\omega^2
\end{equation}
where $\omega$ is a primitive cubic root of $1$ and $i^2=-1$.
\end{lemma}

\begin{proof}
By Table \ref{tab:3} we have $q=p\ne 2,3$. Moreover $M=ZM'$ with $M'\cong \Sp_4(3)$.
The set of orders of elements in $\PSp_4(3)$ is $\{1,2,3,4,5,6,9,12\}$.
By Lemma \ref{powers} and the assumption that $(xy)^5$ is non-scalar, the projective image $\bar x\bar y$
of $xy$ can only have order $9$ or $12$.

Assume first that $\bar x\bar y$ has order $9$. Then, $\bar x\bar y$ is conjugate
both to $(\bar x\bar y)^4$ and to $(\bar x\bar y)^7$. In this case we may assume that
also $xy$ has order $9$. By Remark \ref{rigid triple}, it has $4$ different eigenvalues.
So its Jordan form must be $i^h\cdot \diag(\alpha,\alpha^4, \alpha^6,\alpha^7)$,
where $\alpha$ is a primitive $9$-th root of $1$. It follows that
$\tr{xy}=dr_4=i^h\omega^2$, $\tr{(xy)^{-1}} =r_2=i^{-h}\omega$ and $\tr{(xy)^2} =r_4^2+2ds=i^{2h}\omega$, where $\omega=\alpha^3$.
We obtain that $s=0$, $r_2= i^{-h} \omega$, $r_4=d i^h \omega^2$, as in the statement.

Now, suppose  that $\bar x\bar y$ has order $12$. Since $\Sp_4(3)$ does not have elements of order 24,
a scalar multiple of  $xy$ is an element of order 12 in $\Sp_4(3)$,
whose projective image has the same order. 
$\Sp_4(3)$ has 4 classes of such elements. Over $\F_3$, with respect to the 
form $\blockdiag(J,J)$ where $J=$ antidiag$(1,-1)$, they are represented by: 
$$\pm \blockdiag\left(
\left(\begin{array}{cc}
1&1\\
0&1
\end{array}\right),
\left(
\begin{array}{cc}
0&1\\
-1&0
\end{array}\right)
\right),\ \pm\blockdiag\left(
\left(\begin{array}{cc}
1&-1\\
0&1
\end{array}\right),
\left(
\begin{array}{cc}
0&1\\
-1&0
\end{array}\right)
\right) . $$
Each of these representatives is conjugate to its 7-th power in $\Sp_4(3)$, via $\blockdiag(I,\left(
\begin{array}{cc}
1&1\\
-1&1
\end{array}\right) )$.
Recalling that $xy$ must have 4 different eigenvalues, its Jordan form can only be one of the following,
where $\beta$ is a primitive $12$-th root of $1$, and $h=0,1,2,3$:
$$
i^h\diag(\beta,\beta^4,\beta^7,1),\quad i^h\diag(\beta,\beta^5,\beta^{-1},\beta^{-5}),\quad i^h\diag(\beta,\beta^7,\beta^6, \beta^{10}).
$$
The second possibility is excluded, since it has trace $0$.
So $xy$ has characteristic polynomial
$$t^4 + i^{\ell}\beta^2t^3 - i^{2\ell}t^2 - i^{-\ell}(\beta^2 - 1)t + 1,$$
for some $\ell=0,1,2,3$. Comparison with \eqref{car xy} gives:
$$dr_4= i^\ell\beta^2,\quad ds= i^{2\ell}, \quad r_2= i^{-\ell}(1-\beta^2).$$
If $s=1$, then $\epsilon=\beta^{\pm 2}$; if $s=-1$, then $\epsilon=\beta^{\pm 4}$. It follows that
$H$ is reducible by Lemma \ref{newreducible}(i).
\end{proof}

\begin{lemma}\label{A6}
Assume $p\neq 2$, $d=-1$ and $r_2=-r_4$. If $H$ is absolutely irreducible and contained in a subgroup $M$ such that $M/Z\in \{\Alt(6),\Sym(6),\Alt(7)\}$, then $p=q=7$, $r_4=\pm 4$, $s=-1$ and $H/Z\cong \PSL_2(7)$.
\end{lemma}

\begin{proof}
First, suppose that $(xy)^5$ is scalar. Then, $r_4=\pm 1$  by Lemma \ref{powers}. However, in this case $(xy^2)^h$ is not scalar for all $1\leq h\leq 7$, whence $H$ cannot be contained in a subgroup $M$ such that $M/Z\in \{\Alt(6),\Sym(6),\Alt(7)\}$. So, we may assume that $(xy)^5$ is not scalar. By Lemma \ref{powers}, $M/Z$ must be isomorphic to $\Alt(7)$. Thus, the statement follows from Lemma \ref{A7}.
\end{proof}

\begin{lemma}\label{cubics}
Let $k=3$ (i.e., $s=-1$) and $p\neq 2,3$. Assume $d=-1$ and $r_2=-r_4$. Let
$\phi: \mathrm{SL}_2(q) \rightarrow \mathrm{SL}_4(q)$ be the
homomorphism induced by the action of $\mathrm{SL}_2(q)$ on cubic
polynomials in two variables. The group $H$ is
conjugate to the image, under $\phi$, of some subgroup of
$\mathrm{SL}_2(q)$ if
and only if $r_4^4=-3$.
\end{lemma}

\begin{proof}
1. Sufficiency. Let us consider
$$
 x_2 =\left(
   \begin{array}{rr}
     0 & -1 \\
     1 &  0
   \end{array}
   \right),
   \qquad
 y_2 =\left(
   \begin{array}{rr}
     -1/2 & 3/(2r) \\
     -r/2 & -1/2
   \end{array}
   \right)
$$
for $0\neq r=r_4\in \mathbb{F}_q$.
These matrices act on the graded algebra $\mathbb{F}_q[t_1,t_2]$
as follows:
\begin{eqnarray*}
 & t_1 \stackrel{x_2}{\mapsto} t_2, \quad
 & t_2 \stackrel{x_2}{\mapsto} -t_1, \\
 & t_1 \stackrel{y_2}{\mapsto} -t_1/2-rt_2/2, \quad
 & t_2 \stackrel{y_2}{\mapsto} 3t_1/(2r)-t_2/2.
\end{eqnarray*}
The restriction of this action to the subspace of cubic polynomials (with the standard basis $t_1^3$, $t_1^2t_2$, $t_1t_2^2$, $t_2^3$)
gives us
$$
 \phi(x_2) =\left(
   \begin{array}{rrrr}
 0 & 0 & 0 &-1 \\
 0 & 0 & 1 & 0 \\
 0 &-1 & 0 & 0 \\
 1 & 0 & 0 & 0 \\
   \end{array}
   \right),
   \qquad
 \phi(y_2) =\left(
   \begin{array}{rrrr}
     -1/8  &   3/(8r)  & -9/(8r^2) & 27/(8r^3) \\
   -3r/8   &    5/8    & -3/(8r)   &-27/(8r^2) \\
  -3r^2/8  &   r/8     &     5/8   &    9/(8r) \\
   -r^3/8  &   -r^2/8  &    -r/8   &     -1/8
   \end{array}
   \right).
$$
Let
$$
 Q =\left(
   \begin{array}{rrrr}
       3   &      -3  &    r  &    r \\
   3 r^3 &  3 r^3 & 3 r^2 & -3 r^2 \\
  -3 r^2 &  3 r^2 & 3 r^3 &  3 r^3 \\
      r  &    r   &    -3   &    3
   \end{array}
   \right).
$$
We have $\det Q = 2^6\cdot 3^2 \cdot r^6$, which is non-zero under
our assumptions. 
A direct calculation shows that $Q^{-1}\phi(x_2) Q =x$ and
\begin{eqnarray*}
 && Q^{-1}\phi(y_2) Q - y = \\
 &&  = \frac{r^4+3}{64r^5}
  \left(
  \begin{array}{rrrr}
  -3 r^5 + 8 r^3 - 9 r   &  -3 r^5 + 8 r^3 - 9 r &
    -7 r^4 - 24 r^2 + 27   &  -r^4 + 24 r^2 - 27 \\
  -3 r^5 - 8 r^3 - 9 r   &  -3 r^5 - 8 r^3 - 9 r  &
    -7 r^4 + 24 r^2 + 27   &  -r^4 - 24 r^2 - 27 \\
   3 r^6 + 9 r^2  & 3 r^6 + 9 r^2 &  7 r^5 - 27 r &  r^5 + 27 r \\
  -3 r^6 + 15 r^2 &  -3 r^6 + 15 r^2 &  -7 r^5 - 45 r &  -r^5 + 45r
   \end{array}
   \right).
\end{eqnarray*}
In particular, if $r^4=-3$, then $Q^{-1}\phi(y_2) Q = y$, which proves sufficiency.

2. Necessity. Let $g\in \mathrm{SL}_2(q)$. If $g$ is semisimple, we denote its eigenvalues by $\eta$, $\eta^{-1}$. If $g$ is not semisimple,  let $\eta=\pm1$ be the only root of its characteristic polynomial. In both cases we can write the characteristic polynomial of $\phi(g)$ as
 \begin{eqnarray*}
   && (t-\eta^3)(t-\eta)(t-\eta^{-1})(t-\eta^{-3}) =
    t^4-
    (\eta^3 + \eta + \eta^{-1} + \eta^{-3}) t^3 + \\
   && \qquad\qquad
    (\eta^4 + \eta^2 + 2 + \eta^{-2} + \eta^{-4})t^2 -
    (\eta^3 + \eta + \eta^{-1} + \eta^{-3}) t +
    1.
 \end{eqnarray*}
On the other hand, the characteristic polynomial of $xy$ is
 $$
 t^4 + r t^3 - t^2 + r t + 1.
 $$
In particular, if $xy$ is conjugate to $\phi(g)$ for some $g$, then
 $$
  \left\{
   \begin{array}{l}
      r= \eta^3 + \eta + \eta^{-1} + \eta^{-3}, \\
      \eta^4 + \eta^2 + 3 + \eta^{-2} + \eta^{-4} =0.
   \end{array}
  \right.
 $$
Therefore,
 \begin{eqnarray*}
  r^4+3 &=&  (\eta^3 +\eta +\eta^{-1}+\eta^{-3})^4+3 =
   (\eta^4 + \eta^2 + 3 + \eta^{-2} + \eta^{-4})\times \\
    && \times (\eta^8 + 3\eta^6 + 4\eta^4 + 6\eta^2 + 8 + 6\eta^{-2} + 4\eta^{-4} + 3\eta^{-6} + \eta^{-8}) =0.
 \end{eqnarray*}
This completes the proof.
\end{proof}

\begin{lemma}\label{PSL2}
Assume $d=-1$, $r_2=-r_4$ and $H\leq M$, with $M/Z\cong \PSL_2(q)$.  Then  $p\neq 2,3$, $k=3$   and $r_4^4=-3$.
\end{lemma}

\begin{proof}
By the table at the beginning of this Section we may assume $p\ge 5$. Let $\phi: \mathrm{SL}_2(q) \rightarrow \mathrm{SL}_4(q)$ be the homomorphism
induced by the action of $\mathrm{SL}_2(q)$ on cubic polynomials in two variables $t_1,t_2$.
Then, $M=Q^{-1}\phi(SL_2(q))Q$ for some $Q\in GL_4(q)$ by \cite{M} (see also \cite[Theorem 3.8]{Ki}).
Let $\overline y\in \mathrm{SL}_2(q)$ be such that $\phi(\overline y)=Q^{-1}yQ$.
Assume first that $\overline y$ is semisimple, with eigenvalues $(\alpha, \alpha^{-1})$ over $\F$.
Then $\phi(\overline y)$ has eigenvalues $(\alpha^3, \alpha,\alpha^{-1}, \alpha^{-3})$.
Imposing that two of them are $1$,  we get that $y$ has order $3$.
Next, assume that $\overline y$ is conjugate to
$\left(\begin{array}{cc}
\alpha&1\\
0&\alpha
\end{array}
\right)$, where $\alpha=\pm 1$. Then $\phi(\overline y)$ has the unique eigenvalue $\alpha^3$, whence we get the condition $\alpha=1$.  In this case the eigenspace of $\phi(\overline y)$ relative to $1$ has dimension 1, a contradiction.
Our conclusion follows from Lemma \ref{cubics}.
\end{proof}

\section{Positive results}\label{results}

We recall  that $\F$ is an algebraically closed field of characteristic $p>0$ and $3\leq k\in \N$.  If $(k,p)=1$, then $\epsilon\in \F$ has order $k$.
If $k=p$, $2p$, then   $\epsilon=1$,  $\epsilon=-1$ respectively.
We set $H=\langle x,y \rangle$ with $x,y$ defined as in \eqref{generators},
with $d=\pm 1$,  $s=\epsilon +\epsilon^{-1}$, $r_1=r_3=0$, $r_2\in \F$,
$0\ne r_4\in \F$, i.e.:
\begin{equation}\label{generators2}
x= \left(
\begin{array}{cccc}
0&0&1&0\\
0&0&0&1\\
d&0&0&0\\
0&d&0&0
\end{array}\right), \quad
y=\left(
\begin{array}{cccc}
1&0&0&r_2\\
0&1&0&r_4\\
0&0&0&-1\\
0&0&1&s
\end{array}\right).
\end{equation}
For a fixed $k$, in Table ~\ref{tab:4} we summarize the results from Sections ~\ref{sec:further}--\ref{S},
and describe all the exceptional values of $r_2,r_4$ for which $H$ may be 
contained in some maximal subgroup of the finite classical group under consideration.
The values which correspond to the subfield subgroups do not appear in Table ~\ref{tab:4};
see Lemma \ref{minimal field} instead.

\begin{table}[!h]
\begin{tabular}{cccc}
References& $p$ & Conditions& $HZ/Z$\\
\hline
Corollary \ref{newreducible}&any & $r_2=-\epsilon^{\pm 1} r_4$ & reducible \\
{} &any & $r_2 + r_4=\pm (2- s)\sqrt d$ &  reducible \\  \hline
Lemma \ref{powers} &any & $s=-1$, $r_4^2=d$, $r_2=r_4$ & $\cong \Alt(5)$\\
{} &any& $s=1$, $r^2=-d$,\ $r_2=-r_4$  & $\le CPSp_4(\F)$\\ \hline
Lemma \ref{cond C6} &$\ne 2$ & $s=-1$, $r_2=r_4=\pm \sqrt d$ & $\le M\in {\cal C}_6$ \\
&$\ne 2$& $s=-1$, $r_2=r_4=\pm \sqrt{2d}$ & $\le M\in {\cal C}_6$\\ \hline
Lemma \ref{A7}&$\ne 2$& $s=-1$, $ds=i^{2h}$, $r_2=-i^{3h}\left(\pm \sqrt{-7}-1\right)/2$, & $\le M\in {\mathcal S}$\\
&& $r_4=d i^{h}\left(\pm \sqrt{-7}+1\right)/2$& {} {}\\
\hline
Lemma \ref{PSp(4,3)}&$\ne 2,3$& $s=0$, $r_2  = i^{-h}\omega, r_4 = d i^h\omega^2$, & $\le M\in {\mathcal S}$  \\
{} && $\omega$ a primitive $3^{rd}$ root of $1$ & {}{} \\\hline
Lemma \ref{PSL2} &$\ne 2,3$&  $d=-1$, $s=-1$, $r_2=-r_4$, $r_4^4=-3$ & $\le M\in {\mathcal S}$\\
\hline
\end{tabular}
\caption{{}\quad }  
 \label{tab:4}

\end{table}

For any power $q=p^a$, we have $\F_q\le \F$. It is important to note that,
whenever $k\mid(q-1)$ or $k\mid(q+1)$ or $k=p$ or $k=2p$, then $s\in \F_q$.

\begin{theorem}\label{sl4}  Let $\F_q=\F_p\left[s, r_4^2\right]$, with
$0\neq r_4\in \F_q$.  Define $x$ and $y$ as in \eqref{generators2},
with $r_2=0$, $r_4\ne \pm (s-2)\sqrt d $. Under these assumptions $H=\langle x,y \rangle = \SL_4(q)$.
In particular, for all $k\ge 3$ such that $k\mid (q-1)$ or $k\mid(q+1)$ or $k\in \left\{p,\ 2p\right\}$,
the groups $\SL_4(q)$, $q>3$, and  $\PSL_4(q)$, $q>2$, are $(2,k)$-generated.

Moreover $\SL_4(2)$ is $(2,4)$-generated and 
$\SL_4(3)$ is $(2,3)$ and $(2,6)$-generated.
\end{theorem}
\begin{proof}
By Lemma \ref{minimal field} the group $H$ is not conjugate to a
subgroup of $\SL_4(q_0)$ for any $q_0<q$. By  Theorem \ref{classic}(iii),
$H$ is not contained in the normalizer of any classical subgroup of
$\SL_4(q)$. By the results of Sections from \ref{sec:further} to
\ref{S}, $H$ is not contained in any maximal subgroups of $\SL_4(q)$.
Thus $H=\SL_4(q)$.

Now let $q$ and $s$ be given, with $q>3$ and $s$ as in the statement.
We claim that there exists $r_4\ne 0, \pm (s-2)\sqrt d$ such that
$\F_q=\F_p\!\left[r_4^2\right]$. This is clear when $q=p\geq 5$. If $q=p^a$ with
$a>1$ we use  Lemma \ref{subfields}. Namely, when $q=p^2$ with $p\geq
3$, then our claim follows from it and the inequality
$p^2-2(p-1)-3>0$. When $q=p^a$ with $a\geq 3$, our claim follows from
the inequality
$$p^a-N-3\geq p^a-p \left(p^{\left\lfloor a/2\right\rfloor}-1\right)-3 \geq
p^{\left\lfloor a/2\right\rfloor +1}(p-1)+p-3>0.$$
We are left with the cases $q=4$ and
$(q,s,d)\in\left\{(2,0,1),(3,1,\pm 1),(3,-1,\pm 1),(3,0,-1)\right\}$.

If $q=4$, there exists $r_4\ne 0,1, (s-2)\sqrt d=s$. In the remaining
cases, except $(q,s,d)=(3,1,1)$, we may take $r_4=\pm 1$. Finally, if
$(q,s,d)=(3,1,1)$ the $(2,6)$-generation of $\SL_4(3)$ follows from
the $(2,3)$-generation. Indeed $\left\langle x,y\right\rangle =
\SL_4(3)$ gives $\left\langle x,-y\right\rangle \leq \left\langle
x,y\rangle \langle -I\right\rangle =\SL_4(3)$, whence $\left\langle
x,-y\right\rangle =\SL_4(3)$, since this group is perfect.
Alternatively, a MAGMA calculation shows that $\SL_4(3)$ is generated
by $x,y$ as in \eqref{generators}  with $d=1$, $s=1$, $r_1=-1$,
$r_2=r_3=0$, $r_4=1$.
\end{proof}

\begin{theorem}\label{sp4}  Let $\F_q=\F_p\!\left[s, r_4^2\right]$, with
$0\neq r_4\in \F_q$. Define $x$ and $y$ as in \eqref{generators2},
setting $d=-1$, and $r_2=-r_4\neq 0$. Assume $p\neq 2$, $k\neq p$.
If $k=3$ and $p\neq 3$, assume further that $r_4^4\neq-3$.
Under these assumptions $H=\langle x,y \rangle = \Sp_4(q)$.
In particular, for all $k\geq 3$ such that $k\mid (q-1)$ or $k \mid (q+1)$ or $k=2p$,
the groups $\PSp_4(q)$, with $q$ odd,  are $(2,k)$-generated.
\end{theorem}

\begin{proof}
By Theorem \ref{classic_positive}(v), $H$ is contained in $\Sp_4(q)$.
By Lemma \ref{minimal field} the group $H$ is not conjugate to a
subgroup of $\CSp_4(q_0)\F^\ast I$ for any $q_0<q$. Now we
analyze the conditions implied by the results of
Sections~\ref{sec:further}--\ref{S} (see Table~\ref{tab:4}). Notice
that Corollary~\ref{newreducible} may give exceptional values for
$r_4$ only if $\epsilon=1$ or $s=2$, i.e., when $k=p$, which is
excluded by our assumptions. Since $r_2=-r_4$, Lemma~\ref{cond C6}
and Lemma~\ref{PSp(4,3)} do not give extra conditions. Since $d=-1$,
the conditions given by Lemma~\ref{A7} imply  that $i^{2h}=1$, hence
$r_2=-r_4$ only when $p=7$. In particular $r_4=\pm 3$, values which are
excluded in the statement. By the same reason,
conditions given  by Lemma~\ref{cubics} are excluded. 
We conclude that $H$ is not contained in any maximal subgroup of $\Sp_4(q)$. Thus, $H=\Sp_4(q)$.

Now let $q$ and $s$ be given, with odd $q=p^a$  and $k$ as in the
statement. We claim that it is always possible to find $r_4\neq 0$,
which satisfies the further assumptions of the theorem.

If $k=3$ and $p\neq 3$, let $N_1$ be the number of $r_4\in \F_q$ such that $r_4^4=-3$. Otherwise, let $N_1=0$.
Now our claim is obvious for $q=p\ge 3$, since $p-1>N_1$ in these cases.

If $q=p^a$ with $a>1$, we use Lemma~\ref{subfields} and the number
$N$ defined  therein. For $q=p^2$, our claim follows from the
inequality $p^2-1-N-N_1\ge p^2-2(p-1)-1-N_1>0$, which is valid for
any odd $p$. If $a\ge 3$, our claim follows from
$$p^a-1-N-N_1 \geq p^a-p \left(p^{\left\lfloor a/2\right\rfloor}-1\right)-1-N_1 \geq
p^{\left\lfloor a/2\right\rfloor +1}(p-1)+p-1-N_1>0.$$

\end{proof}


Finally, we consider the unitary case.

\begin{lemma}\label{(2,6)}
Let $x,y$ be as in \eqref{generators}, with $d=1$, $s=1$, $r_1=r_2=\xi$, $r_3=\xi^7$, $r_4=0$,
where $\xi\in \F_9$ is such that $\xi^2-\xi-1=0$. Then $(x,y)$ is a $(2,6)$-generating pair for $\SU_4(9)$.
\end{lemma}

\begin{proof}
$\left\langle x,y\right\rangle$ is absolutely irreducible by Lemma \ref{reducible}.
Using the rigidity of the triple $(x,y,xy)$, \eqref{general car xy} and \eqref{general car xy^-1},
it follows from {\textrm{(ii)}} of Theorem \ref{forms} that $\left\langle x,y\right\rangle\leq \SU(4,9)$.
Calling $\overline a, \overline b$ the projective images of $x^y$ and $(y^2x)^3$ respectively,
direct calculation shows that $\overline a$ and $\overline b$ satisfy the presentation of $\PSL_3(4)$ 
given in Lemma \ref{pres}(iv). The only  maximal subgroups
of $\PSU_4(9)$ whose order is divisible by $7$ belong to the class $\mathcal{S}$ and
are isomorphic either to $A_7$, to $\PSL_3(4)$ or to $\PSU_3(9)$ (see \cite{A}). It follows that
$\left\langle \overline a, \overline b\right\rangle\cong \PSL_3(4)$ is maximal in $\PSU_4(9)$.
Finally let $w=(xy)^2 (xy^2)^4 (xy^5)^2 y^3$. Then $w^9$ is scalar, of order 4.
Since $\PSL_3(4)$ does not have elements of order $9$, we conclude that
$\langle x,y\rangle =\SU_4(9)$.
\end{proof}

\begin{theorem}\label{su4}
Let $s\in \F_q$, $r_4\in \F_{q^2}$ and
 \begin{equation}
  \label{eq:su4qr4}
  \F_{q^2}=\F_p\!\left[ r_4^2\right].
 \end{equation}
  Define $x$ and $y$ as in
\eqref{generators2}, setting $r_2=dr_4^q$. Assume that $q\neq 3$.
Suppose further that

{\rm (i)} $r_4^{q-1} \neq -d \epsilon^{\pm1}$;

{\rm (ii)} $r_4+dr_4^{q} \neq \pm \sqrt{d}(2-s)$;

{\rm (iii)} if $q=p\equiv 3,5,6 \pmod{7}$ and $s=-1$, then for $h=0,1,2,3$ 
$$(r_2,r_4)\neq  \left(-i^{3h} (\lambda \sqrt{-7}-1)/2, di^h
(\lambda\sqrt{-7}+1)/2)\right),\ \lambda=\pm 1 ;$$

{\rm (iv)} if $q=p\equiv 5 \pmod{6}$ and $s=0$, then $(r_2,r_4)\neq
(i^{-h}\omega, di^h\omega^2)$, where $\omega$ is a primitive cubic
root of 1.

Then $H=\langle x,y\rangle= \SU_4(q^2)$.
\end{theorem}

\begin{proof}
By~\eqref{eq:su4qr4}, $r_4\neq 0$. Hence, assumptions (i)--(ii)
together with Corollary~\ref{newreducible} imply that $H$ is
absolutely irreducible. By Theorem~\ref{classic_positive}(i), $H\le
\SU_4(q^2)$. By Lemma \ref{minimal field} the group $H=\left\langle
x,y\right\rangle$ is not conjugate to a subgroup of
$\SL_4(q_0)\F^\ast I$ for any $q_0<q^2$.

Notice that $H$ cannot be a subgroup of the groups described in
Lemma~\ref{cond C6}, as for these cases we would have $r_2=r_4$ hence
$r_4\in \F_q$ in contrast with~\eqref{eq:su4qr4}. Thus, the analysis
made in Sections~\ref{sec:further}--\ref{C6} shows that, if $H\neq
\SU(4,q^2)$, then it can be only a subgroup of a maximal group $M$
from the class $\mathcal{S}$.

Since we assume that $q\neq 3$, $M/Z\not \cong \mathrm{PSL}_3(4)$. Thus,
according to Table~\ref{tab:3}, it remains to consider $M$ with
$M/Z\cong \mathrm{Alt}(7)$ for $q=p\equiv 3,5,6 \pmod{7}$ and
$M/Z\cong \mathrm{PSp}_4(3)$ for $q=p\equiv 5 \pmod{6}$.

By Lemma~\ref{powers}(ii), $(xy)^5$ cannot be scalar, since this may
happen only if $r_4^4=1$, but in that case $r_4^2=\pm1$ in contrast
with~\eqref{eq:su4qr4}. Thus we may apply Lemmas~\ref{A7}
and~\ref{PSp(4,3)}. But these cases are excluded by our assumptions
(iii) and (iv), respectively. Therefore, $H=\langle x,y\rangle=
\SU_4(q^2)$.
\end{proof}

\begin{theorem}
 \label{thm:su4-2}
The groups $\SU_4(q^2)$ and $\PSU_4(q^2)$ are $(2,k)$-generated for
all $k\ge 3$ such that $k\mid (q-1)$ or $k\mid (q+1)$ or $k=p$ or
$k=2p$, except $(q,k)= (2,3)$, $(3,3)$, and $(3,4)$.
\end{theorem}

\begin{proof}
For any fixed $q$ and $k$ (i.e.,  $s$ is also fixed) we count the
number of non-zero $r_4$'s that satisfy the conditions of
Theorem~\ref{su4}. Let $N$ be the number of elements $r_4\neq 0$ in
$\F_{q^2}$ such that $\F_p\!\left[r_4^2\right] \ne \F_{q^2}$ and let $N_1$ be the
number of those $r_4$'s that do not satisfy conditions (i)--(iv) of
Theorem~\ref{su4}. Conditions (iii) and (iv) may give at most 8
exceptions each, but they concern different $s$. Thus, $N_1\le 4q+6$
for any fixed $k$. Moreover, if $q=p^a$ with $a>1$, then we have
$N_1\le 4q-2$ since  cases (iii) and (iv) do not arise.

If $q=p$, then using Lemma~\ref{subfields} for the field $\F_{q^2}$
we have
 $$
  q^2-1-N-N_1 \ge p^2-1-(4p+6)-2(p-1)=p^2-6p-5>0
 $$
if $p\ge 7$. If $q=p^a$ with $a>1$, then using Lemma~\ref{subfields} for the field
$\F_{q^2}$ we have
 $$
  q^2-1-N-N_1 \ge p^{2a}-1-(p^{a+1}-p)-(4p^a-2)
   =(p^a-4)(p^a-p)-3p+1>0
 $$
provided $p^a\ge 8$.

Thus, only  $q=2$, $3$, $4$, and $5$ are left.

If $q=5$, let $\alpha\in\F_{25}$ satisfy $\alpha^2-\alpha+2=0$. In
Table~\ref{tab:5}, for each admissible $k$ and $d$, we list $b$ such
that $r_4=\alpha^b$ satisfies the conditions of Theorem~\ref{su4}.
Since in each case the set of such $b$'s is non-empty, this proves
our Theorem also for $q=5$.

If $q=4$, let $\alpha\in\F_{16}$ satisfy $\alpha^4+\alpha+1=0$. In
particular $\omega=\alpha^5$ is a cubic root of 1. It is enough to
consider $d=1$. In Table~\ref{tab:5}, for each admissible $k$  we
list $b$ such that $r_4=\alpha^b$ satisfies the conditions of
Theorem~\ref{su4}. Since in each case the set of such $b$'s is
non-empty, this proves our Theorem also for $q=4$.

For $q=2$, $k=4$, notice that $r_4=\omega$ and  $r_4=\omega^2$, where
$\omega$ is a primitive cubic root of 1, satisfy the conditions of
Theorem~\ref{su4}. Alternatively, one can use the isomorphism
$\SU_4(4)\cong\PSp_4(3)$ and Theorem~\ref{sp4}.

Finally, for $q=3$, $k=6$, Theorem~\ref{sp4} does not produce
suitable generators of shape \eqref{generators2}, but our claim
follows from Lemma~\ref{(2,6)}.
\end{proof}

\begin{table}[ht]
  
\begin{tabular}{r|r|r|r|l}
 $q$ & $d$ & $k$ & $s$ & $b$ \\ \hline{}
5&  1 & 4 & 0  & 1,4, 5,8, 13,16,17,20 \\
  &  & 3 & $-1$ & 4, 8, 16, 20 \\
   & & 6 &  1 & 7, 11, 19, 23 \\
   & & 5 &  2 & 1, 2, 4, 5, 7, 8, 10, 11, 13, 14, 16, 17, 19, 20, 22, 23 \\
   & &10 & $-2$ & 2, 7, 10, 11, 14, 19, 22, 23 \\
\hline
 5 &  $-1$ & 4 &  0 & 1, 2, 5, 7, 10, 11, 13, 14, 17, 19, 22, 23 \\
   & & 3 & $-1$ & 7, 11, 19, 23 \\
   & & 6 &  1 & 2, 4, 8, 10, 14, 16, 20, 22 \\
   & & 5 &  2 & 1, 2, 4, 5, 7, 8, 10, 11, 13, 14, 16, 17, 19, 20, 22, 23 \\
   & &10 & $-2$ & 1, 2, 4, 5, 7, 8, 10, 11, 13, 14, 16, 17, 19, 20, 22, 23\\\hline \hline
  4& 1 & 3 & $-1$       & 3, 6, 7, 9, 11, 12, 13, 14 \\
  &&  5 & $\omega$   & 1, 4, 11, 14 \\
  &&  5 & $\omega^2$ & 2, 7, 8, 13 \\
   && 4 & 0          & 1, 2, 3, 4, 6, 7, 8, 9, 11, 12, 13, 14
\end{tabular}

\caption{{}\quad }  \label{tab:5}

\end{table}

 \end{document}